\documentclass{amsart}

\usepackage{latexsym}

\usepackage{amsfonts,amssymb,euscript,epsfig,color}

\usepackage[all]{xy}

\newtheorem{theo}{Theorem}[section]
\newtheorem{defi}[theo]{Definition}

\newtheorem{lem}[theo]{Lemma} 
\newtheorem{prop}[theo]{Proposition}

\newtheorem{rem}[theo]{Remark}
\newtheorem{coro}[theo]{Corollary}

\newtheorem{exam}[theo]{Example}
\newtheorem{notation}[theo]{Notation}

\newcommand{\agot}{\ensuremath{\mathfrak{a}}}
\newcommand{\ggot}{\ensuremath{\mathfrak{g}}}
\newcommand{\hgot}{\ensuremath{\mathfrak{h}}}
\newcommand{\kgot}{\ensuremath{\mathfrak{k}}}
\newcommand{\lgot}{\ensuremath{\mathfrak{l}}}
\newcommand{\rgot}{\ensuremath{\mathfrak{r}}}

\newcommand{\pgot}{\ensuremath{\mathfrak{p}}}
\newcommand{\slgot}{\ensuremath{\mathfrak{sl}}}
\newcommand{\tgot}{\ensuremath{\mathfrak{t}}}
\newcommand{\ugot}{\ensuremath{\mathfrak{u}}}
\newcommand{\Rgot}{\ensuremath{\mathfrak{R}}}

\newcommand{\Acal}{\ensuremath{\mathcal{A}}}

\newcommand{\Ccal}{\ensuremath{\mathcal{C}}}
\newcommand{\Dcal}{\ensuremath{\mathcal{D}}}

\newcommand{\Hcal}{\ensuremath{\mathcal{H}}}
\newcommand{\Kcal}{\ensuremath{\mathcal{K}}}
\newcommand{\Lcal}{\ensuremath{\mathcal{L}}}
\newcommand{\Qcal}{\ensuremath{\mathcal{Q}}}

\newcommand{\Ocal}{\ensuremath{\mathcal{O}}}

\newcommand{\Xcal}{\ensuremath{\mathcal{X}}}

\newcommand{\Ucal}{\ensuremath{\mathcal{U}}}
\newcommand{\Vcal}{\ensuremath{\mathcal{V}}}

\newcommand{\Z}{\ensuremath{\mathbb{Z}}}
\newcommand{\C}{\ensuremath{\mathbb{C}}}

\newcommand{\R}{\ensuremath{\mathbb{R}}}
\newcommand{\N}{\ensuremath{\mathbb{N}}}

\newcommand{\h}{{\ensuremath{\rm h}}}
\newcommand{\ad}{{\ensuremath{\rm ad}}}
\newcommand{\As}{{\ensuremath{\rm As}}}
\newcommand{\AS}{{\ensuremath{\rm AS}}}

\newcommand{\f}{\ensuremath{\mathcal{C}^{\infty}}}
\newcommand{\fgene}{\ensuremath{\mathcal{C}^{-\infty}}}
\newcommand{\croc}{\ensuremath{\hookrightarrow}}

\newcommand{\indice}{\ensuremath{\hbox{\rm Index}}}

\newcommand{\mm}{\ensuremath{\hbox{\rm m}}}
\newcommand{\Cr}{\ensuremath{\hbox{\rm Cr}}}
\newcommand{\Vol}{\ensuremath{\hbox{\rm vol}}}
\newcommand{\T}{\ensuremath{\hbox{\bf T}}}
\newcommand{\Tr}{\ensuremath{\hbox{\rm Tr}}}
\newcommand{\K}{\ensuremath{\hbox{\bf K}}}

\newcommand{\Char}{\ensuremath{\hbox{\rm Char}}}
\newcommand{\End}{\ensuremath{\hbox{\rm End}}}
\newcommand{\Thom}{\ensuremath{\hbox{\rm Thom}}}

\newcommand{\spp}{\ensuremath{\underline{\perp}}}

\newcommand{\qfor}{\ensuremath{\mathcal{Q}^{-\infty}}}
\newcommand{\Rfor}{\ensuremath{R^{-\infty}}}
\newcommand{\Rforc}{\ensuremath{R^{-\infty}_{tc}}}
\newcommand{\Chol}{\ensuremath{\Ccal_{\rm hol}}}
\newcommand{\chol}{\ensuremath{\Ccal^{^{\geq}}_{\rm hol}}}

\def \what {\widehat}

\def \cay {{\rm \bf c}}

\def \U {{\rm U}}
\def \SU {{\rm SU}}
\def \S {{\rm S}}

\def \Clif {{\rm Cl}}

\def \indT {{\rm Ind}^{K}_{T}}
\def \indl {{\rm Ind}^{K}_{K_{\Lambda(\lambda)}}}
\def \indL {{\rm Ind}^{K}_{K_{\Lambda}}}
\def \HolT {{\rm Hol}^{K}_{T}}

\def \m {{\rm \tiny min}}

\title{Multiplicities of the discrete series}

\author{Paul-Emile  PARADAN}

\address{Institut de Math\'ematiques et Mod\'elisation de Montpellier (I3M), 
Universit\'e Montpellier 2} 

\email{Paul-Emile.Paradan@math.univ-montp2.fr}

\date{November 2008}

\begin{document}

\begin{abstract} 
The purpose of this paper is to show that the multiplicities of a discrete series 
representation relatively to a compact subgroup can be ``computed'' geometrically, 
in the way predicted by the ``qantization commutes with reduction'' principle of 
Guillemin-Sternberg.
\end{abstract}

\maketitle

{\def\thefootnote{\relax}
\footnote{{\em Keywords} : moment map, reduction, geometric quantization, 
discrete series, transversally elliptic symbol.\\
{\em 1991 Mathematics Subject Classification} : 58F06, 57S15, 19L47.}
\addtocounter{footnote}{-1}}

{\small
\tableofcontents
}


\section{Introduction and statement of the results}

In the 50's, Harish Chandra has constructed the holomorphic discrete series representations of
a real semi-simple Lie group $G$ as quantization of certain {\em rigged} elliptic orbits, 
in a way similar to the Borel-Weil Theorem. Here the quantization procedure is the one that 
Kirillov calls {\em geometric} in \cite{Kirillov99}.

The first purpose of this paper is to explain how one can ``compute'' geometrically the mutiplicities 
of a holomorphic representation of a real simple Lie group relatively to a compact connected 
subgroup : our main result is Theorem \ref{theo:intro}. This computation follows the line of the 
orbit method \cite{Kirillov99} and is a non-compact example of 
the ``quantization commutes with reduction'' phenomenon 
\cite{Guillemin-Sternberg82,Meinrenken-Sjamaar,pep-RR,Tian-Zhang98}.

Next we show  that this result extends to the discrete series representations  for which 
the Harish-Chandra and Blattner parameters belong to the same chamber of strongly elliptic elements. 
See Theorem \ref{theo:intro-deux}.

In our previous article \cite{pep-ENS}, we prove that a similar result occurs : 
the multiplicities of any discrete series representation relatively  
to a maximal compact subgroup can be ``computed'' geometrically. 

Nevertheless, our present contribution is not a consequence of the results 
of \cite{pep-ENS}, for two reasons:
\begin{enumerate}
\item In \cite{pep-ENS}, we were working with the {\em metaplectic version} 
of the quantization (we prefer the denomination ``Spin'' quantization). 
Here we work with the {\em geometric version} of the quantization. See the review of 
Vogan \cite{vogan-review} for a brief explanation concerning this 
two kinds of quantization.

\item The other main difference with \cite{pep-ENS} is that here we look at 
the multiplicities relatively to {\em any compact connected subgroup}, subordinated to the 
condition that the multiplicities are finite. 
\end{enumerate}

Our main tool to investigate (2) is the ``formal geometric quantization'' 
procedure that we have studied in \cite{pep-formal}.

Finally, we mention that our present paper is strongly related to the works of 
Kobayashi \cite{TK1,TK2,TK3} and Duflo-Vargas \cite{Duflo-Vargas} where they 
study the general setting of restrictions of 
irreducible representations to a reductive subgroup.

\subsection{Realisation of the holomorphic discrete series}

Let $G$ be a connected real simple Lie group with finite center and let $K$ be a maximal compact subgroup. 
We make the choice of a maximal torus $T$ in $K$.  Let $\ggot$, $\kgot$, $\tgot$ 
be the Lie algebras of $G$, $K$, $T$. We consider the Cartan decomposition 
$\ggot=\kgot\oplus\pgot$.

We assume that $G$ admits holomorphic discrete series representations. It is the case if and only if 
the real vector space $\pgot$ admits a $K$-invariant complex structure, or equivalently,  if the center 
$Z(K)$ of $K$ is equal to the circle group : hence the complex structure on $\pgot$ is 
defined by the adjoint action of an element $z_o$ in the Lie algebra of $Z(K)$. 

Let $\wedge^*\subset\tgot^*$ be the weight lattice : $\alpha\in\wedge^*$ if $i\alpha$ is the 
differential of a character of $T$. Let $\Rgot\subset \wedge^*$ be the set of roots 
for the action of $T$ on $\ggot\otimes\C$. We have $\Rgot=\Rgot_c\cup \Rgot_n$ 
where $\Rgot_c$ and $\Rgot_n$ are respectively the set of roots for the action of $T$ on 
$\kgot\otimes\C$ and $\pgot\otimes\C$. We fix a system of positive roots $\Rgot^+_c$ in 
$\Rgot_c$. We have $\pgot\otimes\C=\pgot^+\oplus\pgot^-$ where the $K$-module $\pgot^\pm$ 
is equal to $\ker({\rm ad}(z_o)\mp i)$. 
Let $\Rgot^{\pm,z_o}_n$ be the set of roots for the action of $T$ 
on $\pgot^\pm$. The union $\Rgot^+_c\cup\Rgot^{+,z_o}_n$ defines 
then a system of positive roots in $\Rgot$ that we denote by $\Rgot^{+}_{\rm hol}$.

Let $\tgot^*_+\subset\tgot^*$ be the Weyl chamber defined by the system of positive roots 
$\Rgot^+_c$. Let  $\Chol:=\left\{ \xi\in\tgot^*_+\ \vert\ (\beta,\xi)>0, \ 
\forall \beta\in \Rgot^{+,z_o}_n\right\}$, 
where $(\cdot,\cdot)$ denotes the scalar product on $\tgot^*$ induced by the Killing form of $\ggot$. So 
the closure $\overline{\Chol}$ is the Weyl chamber defined by the system of 
positive roots $\Rgot^+_{\rm hol}$. 

The complex vector space $\pgot^+$ is an irreducible $K$-representation. Hence, if $\beta_\m$ 
is the lowest $T$-weight on $\pgot^+$, every weight $\beta\in\Rgot^{+,z_o}_n$ is of the form 
$\beta=\beta_\m+\sum_{\alpha\in\Rgot^+_c} n_\alpha \alpha$ with $n_\alpha\in\N$. Then we have 
\begin{equation}\label{eq:Chol} 
\Chol=\tgot^*_+\cap \{\xi\in\tgot^*\ \mid \ (\xi,\beta_\m)> 0\}.
\end{equation}
Note that every $\xi \in\Chol$ is {\em strongly elliptic}: the stabilizer subgroup $G_\xi$ is compact and coincides 
with the stabilizer subgroup $K_\xi$.

For every weight $\Lambda\in\wedge^*\cap \Chol$, we consider the coadjoint orbit 
$$
\Ocal_\Lambda:= G\cdot \Lambda\subset \ggot^*.
$$
 For $X\in\ggot$, let 
$VX$ be the vector field on $\Ocal_\Lambda$ defined by : $VX(\xi):=\frac{d}{dt} 
e^{-tX}\cdot \xi\vert_{t=0}$, $\xi\in \Ocal_\Lambda$. We have on the coadjoint 
orbit $\Ocal_\Lambda$ the following data :

\begin{enumerate}
\item The Kirillov-Kostant-Souriau symplectic form $\Omega_\Lambda$ which is 
defined by the relation : for any $X,Y\in\ggot$ and $\xi\in\Ocal_\Lambda$ we have 
$$
\Omega_\Lambda(VX,VY)\vert_\xi=\langle\xi,[X,Y]\rangle.
$$

\item The inclusion $\Phi_G:\Ocal_\Lambda\hookrightarrow \ggot^*$ is a moment map relative 
to the Hamiltonian action of $G$ on $(\Ocal_\Lambda,\Omega_\Lambda)$.

\item A $G$-invariant complex structure $J_\Lambda$ characterized by the 
following fact. The holomorphic tangent bundle 
$\T^{1,0}\Ocal_\Lambda\to \Ocal_\Lambda$ is equal, above 
$\Lambda\in\Ocal_\Lambda$, to the $T$-module
$$
\sum_{\substack{\alpha\in \Rgot_c^+\\ \langle\alpha,\Lambda\rangle\neq 0}}\ggot_\alpha  +
\underbrace{\sum_{\beta\in \Rgot_n^-}\ggot_\beta}_{\pgot^-} \  .
$$

\item The line bundle $\Lcal_\Lambda:=G\times_{K_\Lambda}\C_\Lambda$ over 
$G/K_\Lambda\simeq \Ocal_\Lambda$ with its canonical holomorphic structure. Here $\C_\Lambda$ 
is the one dimensional representation of the stabilizer subgroup $K_\Lambda$ attached to the 
weight $\Lambda\in\wedge^*$.
\end{enumerate}

\medskip

One can check that the complex structure $J_\Lambda$ is positive relatively to the symplectic form, e.g. 
$\Omega_\Lambda(-,J_\Lambda-)$ defines a Riemannian metric on $\Ocal_\Lambda$. Hence 
$(\Ocal_\Lambda,\Omega_\Lambda,J_\Lambda)$ is a  K\"ahler manifold.  Moreover the first Chern 
class of $\Lcal_\Lambda$ is equal to $\bigl[\frac{\Omega_\Lambda}{2\pi}\bigr]$ : the line bundle 
$\Lcal_\Lambda$ is an equivariant {\em pre-quantum} line bundle over $(\Ocal_\Lambda,\Omega_\Lambda)$ 
\cite{Kostant70,Souriau}.

\medskip

We are interested in the geometric quantization of the coadjoint orbits $\Ocal_\Lambda$,  
$\Lambda\in\wedge^*\cap \Chol$. We take on $\Ocal_\Lambda$ the invariant volume form defined 
by its symplectic structure. The line bundle $\Lcal_\Lambda$ is equipped with a $G$-invariant 
Hermitian metric (which is unique up to a multiplicative constant).

\begin{defi}\label{def:quant-lambda} We denote $\Qcal_G(\Ocal_\Lambda)$ 
the Hilbert space of square integrable holomorphic sections of the line bundle  
$\Lcal_\Lambda\to\Ocal_\Lambda$.
\end{defi}

The irreducible representations of $K$ are parametrized by the set of dominant weights, that we denote
$$
\what{K}:=\wedge^*\cap \tgot^*_+.
$$
For any $\mu\in\what{K}$, we denote $V^K_\mu$ the irreducible representation of 
$K$ with highest weight $\mu$.

Let $\rho_n$ be half the sum of the elements of $\Rgot^{+,z_o}_n$. 
Let $S(\pgot^+)$ be the symmetric algebra of the vector space $\pgot^+$ :  
it is an admissible representation  of $K$ since the center $Z(K)$ acts on $\pgot^+$ 
as the rotation group.

The following theorem is due to  Harish Chandra \cite{Harish-Chandra55-56}. See also the nice exposition 
\cite{Knapp-hol}.

\begin{theo}\label{theo:H-C}
Let $\Lambda\in\wedge^*\cap \Chol$. Then 
\begin{itemize}
\item If $(\Lambda,\beta_\m)<2(\rho_n,\beta_\m)$, the Hilbert space $\Qcal_G(\Ocal_\Lambda)$ 
is reduced to $\{0\}$.\\

\item If $(\Lambda,\beta_\m)\geq 2(\rho_n,\beta_\m)$, the Hilbert space $\Qcal_G(\Ocal_\Lambda)$ 
is an irreducible representation of $G$ such that the subspace of $K$-finite vectors is isomorphic 
to $V^K_\Lambda\otimes S(\pgot^+)$.
\end{itemize}
\end{theo}

The holomorphic discrete series representations of $G$ are those of the form 
$\Qcal_G(\Ocal_\Lambda)$, for  $\Lambda\in \what{K}\cap \chol$ where  
\begin{equation}\label{eq:Blattner-par}
\chol=\big\{\xi\in\tgot^*_+\ \mid \ (\xi-2\rho_n,\beta_\m)\geq 0\big\}.
\end{equation}

Here, we have parametrized the holomorphic discrete series representations \break 
$\Qcal_G(\Ocal_\Lambda)$ by their 
{\em Blattner parameter} $\Lambda\in\what{K}\cap \chol$. The corresponding 
{\em Harish-Chandra parameter} is 
$\lambda:=\Lambda+\rho_c-\rho_n$, where $\rho_c$ is half the sum of the elements of 
$\Rgot^+_c$. One checks that the map $\Lambda\mapsto \Lambda+\rho_c-\rho_n$ is a one to one map
between $\what{K}\cap \chol$ and 
\begin{equation}\label{eq:G-hat-hol}
\widehat{G}_{\rm hol}:=\{\lambda\in\tgot^*\ \vert\ (\lambda,\alpha)>0\  
\forall \alpha\in \Rgot^+_{\rm hol}\ {\rm and} 
\ \lambda+\rho_n+\rho_c\in\wedge^*\}.
\end{equation}

\medskip

\begin{exam}\label{ex:sp(2)}
Let us consider the case of the symplectic group $G={\rm Sp}(2, \R)$. 
Here $K$ is the unitary group $\U(2)$, and the maximal torus is of dimension $2$. 
In the figure \ref{figure-SP4}, we draw the chambers $\chol\subset\Chol\subset\tgot^*_+$, 
$\alpha$ is the unique positive compact root, and $\beta_1,\beta_2, \beta_3$ 
are the positive non-compact roots. The root $\beta_3$ corresponds to the 
root $\beta_\m$ used in (\ref{eq:Chol}).
\end{exam}

\begin{figure}[h]\label{figure-SP4}
\centerline{
\scalebox{0.4}{\input{SP4.pstex_t}}
}
\caption{The case of ${\rm Sp}(2,\R)$}
\end{figure}

\subsection{Main results concerning the holomorphic discrete series}\label{sec:intro-hol-DS}

Let $H\subset K$ be a compact connected Lie group with Lie algebra $\hgot$. 
The $H$-action on $(\Ocal_\Lambda,\Omega_\Lambda)$ 
is Hamiltonian with moment map $\Phi_H:\Ocal_\Lambda\to\hgot^*$ equal to the composition  of 
$\Phi_G:\Ocal_\Lambda\to\ggot^*$ with the projection $\ggot^*\to\hgot^*$.

\begin{notation} We denote $\Qcal_H(\Ocal_\Lambda)$ the (dense) vector subspace 
of $\Qcal_G(\Ocal_\Lambda)$ formed by the $H$-finite vectors.
\end{notation}

When $\Lambda\in\chol$, we know thanks to  Theorem \ref{theo:H-C}, that $\Qcal_H(\Ocal_\Lambda)$ is the ``restriction'' of the $K$-representation $V^K_\Lambda\otimes S(\pgot^+)$:  we will also denote it as 
$V^K_\Lambda\otimes S(\pgot^+)\vert_H$. We are interested in the case where the $H$-multiplicities in 
$V^K_\Lambda\otimes S(\pgot^+)\vert_H$ are finite, e.g. $V^K_\Lambda\otimes S(\pgot^+)\vert_H$
is $H$-{\em admisssible}. 

The {\em asymptotic $K$-support} of a $K$-representation $E$ is the closed 
cone of $\tgot^*_+$, denoted by $\AS_K(E)$, formed 
by the limits $\lim_{n\to\infty}\epsilon_n\mu_n$, 
where $(\epsilon_n)_{n\in\N}$ is a sequence of non-negative real numbers converging to $0$ and 
$(\mu_n)_{n\in\N}$ is a sequence of $\what{K}$ such that $\hom_K(V_{\mu_n}^K,E)\neq 0$ for 
all $n\in\N$.

For any closed subgroup $H$ of $K$, we denote $\hgot^\perp\subset\kgot^*$ the orthogonal for the duality 
of the Lie algebra of $H$. We have the following result of T. Kobayashi. 

\begin{prop}[\cite{TK1,TK4}]\label{prop:Toshi}
Let $E$ be an admissible $K$-representation. Let $H$ be a compact subgroup of $K$.
Then the following two conditions are equivalent:
\begin{enumerate}
\item $E\vert_H$ is $H$-admissible.
\item $\AS_K(E)\cap K\cdot \hgot^\perp=\{0\}$.
\end{enumerate}
\end{prop}

Let $\{\gamma_1,\ldots,\gamma_r\}$ be a maximal family of strongly orthogonal roots 
(see Section \ref{sec:delta-K-p}). 
Schmid \cite{Schmid69} has shown that $S(\pgot^+)$ is a $K$-representation without multiplicity, and that 
the representation $V_\mu^K$ occurs in $S(\pgot^+)$ if and only if 
$$
\mu=\sum_{k=1}^r n_k(\gamma_1+\cdots +\gamma_k),\quad {\rm with}\quad n_k\in\N.
$$ 

Thus, we check easily that the asymptotic $K$-support of $V^K_\Lambda\otimes S(\pgot^+)$ is equal to
$$
\sum_{k=1}^r\R^{\geq 0}(\gamma_1+\cdots +\gamma_k).
$$

 The following proposition is proved in Sections \ref{sec:proper} and \ref{sec:delta-K-p} (see Theorem 
\ref{theo:equivalence}).

\begin{prop}\label{prop:admissible} 
Let $\Lambda\in\Chol$. The following statements are equivalent.
\begin{enumerate}
\item The representation $V^K_\Lambda\otimes S(\pgot^+)\vert_H$ is 
{\em admissible}.
\item We have $\sum_{k=1}^r\R^{\geq 0}(\gamma_1+\cdots +\gamma_k)\cap K\cdot\hgot^\perp=\{0\}$.
\item The map $\Phi_H:\Ocal_\Lambda\to\hgot^*$ is {\em proper}.
\end{enumerate}
\end{prop}

We know from Proposition \ref{prop:Toshi} that (1) and (2) are equivalent, thus our main 
contribution is the equivalence with (3). Nevertheless our proof of Proposition 
\ref{prop:admissible} does not use directly the result of Proposition \ref{prop:Toshi}. 
We prove in Section \ref{sec:proper} that (1) and (3) are both equivalent to the condition 
$$
\Delta_K(\pgot)\cap K\cdot\hgot^\perp=\{0\},
$$
where $\Delta_K(\pgot)\subset\tgot^*_+$ is the Kirwan convex set associated to the Hamiltonian action of 
$K$ on $\pgot$. In Section \ref{sec:delta-K-p}, a direct computation gives that 
\begin{equation}\label{intro:delta-K-p}
\Delta_K(\pgot)=\sum_{k=1}^r\R^{\geq 0}(\gamma_1+\cdots +\gamma_k).
\end{equation} 

Another way to obtain (\ref{intro:delta-K-p}) is  by using the theorem of Schmid 
(which computes the $K$-multiplicities in 
$S(\pgot^+)$) together with the following fact: for any affine variety 
$\Xcal\subset \C^n$ which is invariant relative to the linear action of $K$ 
on $\C^n$, the Kirwan set $\Delta_K(\Xcal)$ is equal to the asymptotic $K$-support 
of the algebra $\C[\Xcal]$ of polynomial functions on $\Xcal$ (see the Appendix by Mumford 
in \cite{Ness84}).

\begin{exam}
Since the representation $\Qcal_{Z(K)}(\Ocal_\Lambda)$ is  admissible, the representation
$\Qcal_H(\Ocal_\Lambda)$ will be admissible for any subgroup $H$ containing $Z(K)$.
\end{exam}

The irreducible representations of the compact Lie group $H$ are parametrized by a set of 
dominant weights 
$\widehat{H}\subset\hgot^*$. For any $\mu\in \widehat{H}$, we denote $V^H_\mu$ the 
irreducible representation of $H$ with highest weight $\mu$.  

\medskip

We suppose now that the moment map $\Phi_H:\Ocal_\Lambda\to\hgot^*$ is proper, 
and one wants to compute the multiplicities of 
$\Qcal_H(\Ocal_\Lambda)$. 

If $\xi\in\hgot^*$ is a regular value of $\Phi_H$, the Marsden-Weinstein reduction 
$$
\left(\Ocal_\Lambda\right)_\xi:=\Phi_H^{-1}(H\cdot\xi)/H
$$
is a compact K\"ahler  orbifold. If moreover $\xi$ is integral, e.g. $\xi=\mu\in  \widehat{H}$, 
there exists 
a holomorphic line orbibundle $\Lcal(\mu)$ that prequantizes the symplectic orbifold 
$\left(\Ocal_\Lambda\right)_\mu$. In this situation, one defines the integer 
$$
\Qcal\left((\Ocal_\Lambda)_\mu\right)\in\Z,
$$
as the holomorphic Euler characteristic of $((\Ocal_\Lambda)_\mu,\Lcal(\mu))$. 

In the general case where $\mu$ is not necessarily a regular value of $\Phi_H$, 
$\Qcal\left((\Ocal_\Lambda)_\mu\right)\in\Z$ can still be defined (see \cite{Meinrenken-Sjamaar,pep-RR}). 
The integer $\Qcal\left((\Ocal_\Lambda)_\mu\right)$ only depends on the data 
$(\Ocal_\Lambda,\Lcal_\Lambda,J_\Lambda)$ in a small neighborhood of $\Phi^{-1}_H(\mu)$ : 
in particular $\Qcal\left((\Ocal_\Lambda)_\mu\right)$ 
vanishes when  $\mu$ does not belong to the image of $\Phi_H$.

\bigskip

Now we can state one of our main result.

\medskip

\begin{theo}\label{theo:intro}  Consider a holomorphic discrete series representation $\Qcal_G(\Ocal_\Lambda)$ 
with Blattner parameter $\Lambda\in \chol$. Let $H\subset K$ be a compact connected Lie group 
such that the representation $\Qcal_H(\Ocal_\Lambda)$ is admissible.  Then we have
$$
\Qcal_H(\Ocal_\Lambda)=\sum_{\mu\in\widehat{H}}\Qcal\left((\Ocal_\Lambda)_\mu\right)V^H_\mu.
$$
In other words, the multiplicity of $V^H_\mu$ in the holomorphic discrete series representation 
$\Qcal_G(\Ocal_\Lambda)$ is equal to $\Qcal\left((\Ocal_\Lambda)_\mu\right)$.
\end{theo}

\bigskip

A question still remains. 
When $\mu\in\what{H}$ is a {\em regular} value of the moment map $\Phi_H$, Theorem \ref{theo:intro} says that 
the multiplicity ${\rm m}_\Lambda(\mu)$ of the irreducible representation $V_\mu^H$ in $\Qcal_H(\Ocal_\Lambda)$
is equal to the holomorphic Euler characteristic of line orbibundle $\Lcal(\mu)\to (\Ocal_\Lambda)_\mu$. Does 
the multiplicity ${\rm m}_\Lambda(\mu)$ coincides with the dimension of the vector space 
$$
H^0((\Ocal_\Lambda)_\mu,\Lcal(\mu))
$$
of holomorphic sections of $\Lcal(\mu)\to (\Ocal_\Lambda)_\mu$ ?

\subsection{Main result concerning the discrete series}

We work now with a real semi-simple Lie group $G$ such that a maximal torus 
$T$ in $K$ is a Cartan subgroup of $G$. We know then that $G$ has discrete series representations 
\cite{Harish-Chandra65-66}. Nevertheless, we do not assume that $G$ has {\em holomorphic} 
discrete series representations. 

Harish-Chandra parametrizes the discrete series representations of $G$ by a 
discrete subset $\widehat{G}_{d}$ of regular elements of 
the Weyl chamber $\tgot^*_+$ \cite{Harish-Chandra65-66}.  
He associates to any $\lambda\in \widehat{G}_{d}$ an irreducible, square integrable, unitary 
representation $\Hcal_{\lambda}$ of $G$ : $\lambda$ is the {\em Harish-Chandra parameter} of 
$\Hcal_{\lambda}$. The corresponding {\em Blattner parameter} of $\Hcal_{\lambda}$ is 
$$
\Lambda(\lambda):= \lambda-\rho_c +\rho_n(\lambda)\ \in\ \wedge^*,
$$
where $\rho_n(\lambda)$ is half the sum of the non-compact roots $\beta$ satisfying $(\beta,\lambda)>0$.

We work under the following condition
\begin{equation}\label{condition-lambda}
(\beta,\lambda)(\beta,\Lambda(\lambda))>0\quad {\rm for \ any}\quad \beta\in\Rgot_n.
\end{equation}

The set of strongly elliptic elements of the Weyl chamber $\tgot^*_+$ decomposes as an union 
$\Ccal_1\cup\cdots\cup\Ccal_r$ of connected component : each chamber $\Ccal_i$ corresponds to a choice of 
positive roots $\Rgot^{+,i}\subset\Rgot$ containing $\Rgot_c^+$. Condition (\ref{condition-lambda}) asks that 
$\lambda$ and $\Lambda(\lambda)$ belong to the same chamber $\Ccal_i$. 

When $G$ admits holomorphic discrete series, there is a particular chamber $\Chol$ of strongly elliptic elements such 
that the intersection $\widehat{G}_{d}\cap \Chol$ is equal to the subset $\widehat{G}_{\rm hol}$ defined in 
(\ref{eq:G-hat-hol}). We noticed already that the map $\lambda\mapsto \Lambda(\lambda)$ defines 
a one to  one map between $\widehat{G}_{d}\cap \Chol$ and $\widehat{K}\cap \chol$. In particular, 
any $\lambda\in \widehat{G}_{d}\cap \Chol$ satisfies Condition \ref{condition-lambda}. 
We give in Section \ref{section:condition-lambda} some examples where Condition (\ref{condition-lambda}) does not hold.

Let $\lambda \in \widehat{G}_{d}$ satisfying (\ref{condition-lambda}). The coadjoint orbit $\Ocal_{\Lambda(\lambda)}$ 
is pre-quantized by the line bundle $G\times_{K_{\Lambda(\lambda)}}\C_{\Lambda(\lambda)}$. One of the main difference 
with the  holomorphic case is that the orbit $\Ocal_{\Lambda(\lambda)}$ 
is equipped with an invariant  {\em almost complex structure} $J_{\Lambda(\lambda)}$, which is 
compatible with the symplectic form, but which is not {\em integrable} in general.

Let $H$ be a compact connected Lie subgroup of $K$. Suppose that the moment map,  
$\Phi_H:\Ocal_{\Lambda(\lambda)}\to\hgot^*$ corresponding to the Hamiltonian action of $H$ 
on $\Ocal_{\Lambda(\lambda)}$, is {\em proper}. The reduced 
spaces $(\Ocal_{\Lambda(\lambda)})_\mu:=\Phi_H^{-1}(H\cdot\mu)/H$ are in general {\em not K\"ahler}. 
Nevertheless, their geometric quantization $\Qcal\left((\Ocal_{\Lambda(\lambda)})_\mu\right)\in\Z$ 
are well defined as the index of a Dolbeault-Dirac operator (see \cite{Meinrenken-Sjamaar,pep-RR}).

The following theorem is proved in Section \ref{sec:disc-series}.

\begin{theo}\label{theo:intro-deux}  Consider a discrete series representation $\Hcal_\lambda$ 
with a Harish-Chandra parameter $\lambda\in \widehat{G}_{d}$ satisfying condition (\ref{condition-lambda}). 
Let $H\subset K$ be a compact connected Lie subgroup such that the moment map 
$\Phi_H:\Ocal_{\Lambda(\lambda)}\to\hgot^*$ is {\em proper}. Then

$\bullet$ the representation $\Hcal_\lambda\vert_ H$ is admissible,

$\bullet$ we have 
$$
\Hcal_\lambda\vert_ H=\sum_{\mu\in\widehat{H}}\Qcal\left((\Ocal_{\Lambda(\lambda)})_\mu\right)V^H_\mu.
$$
In other words, the multiplicity of $V^H_\mu$ in the discrete series representation 
$\Hcal_\lambda$ is equal to the quantization of the symplectic reduction $(\Ocal_{\Lambda(\lambda)})_\mu$.
\end{theo}

Theorem \ref{theo:intro-deux} applies for (most of) the discrete series, but is less precise than the 
results described in Section \ref{sec:intro-hol-DS}. We {\em expect} that:
\begin{enumerate}
\item The properness of $\Phi_H:\Ocal_{\Lambda(\lambda)}\to\hgot^*$  should only depend of the chamber $\Ccal_i$ 
containing $\Lambda(\lambda)$.
\item The properness of the the moment map $\Phi_H:\Ocal_{\Lambda(\lambda)}\to\hgot^*$ 
should be equivalent to the admissibility of the restriction $\Hcal_\lambda\vert_H$.
\end{enumerate}

Duflo-Vargas \cite{Duflo-Vargas} have shown that the admissibility of the restriction $\Hcal_\lambda\vert_H$ is 
equivalent to the properness of the moment map $\Ocal_{\lambda}\to\hgot^*$.  Since we assume that 
$\Lambda(\lambda)$ and $\lambda$ belong to the same chamber, point $(1)$ induces point $(2)$. 

Something which is also lacking is an effective criterium which tells us when the map 
$\Phi_H:\Ocal_{\Lambda(\lambda)}\to\hgot^*$ is proper. See \cite{Duflo-Vargas} for some results 
in this direction.

\bigskip

{\bf Acknowledgments.}\ I am  grateful to Michel Duflo and Mich\`ele Vergne for
valuable comments and useful discussions on these topics.

\medskip

\section{Quantization commutes with reduction}\label{QR=0}

In this section, first we recall the ``quantization commutes with reduction''  phenomenon of 
Guillemin-Sternberg which was first proved by Meinrenken and \break Meinrenken-Sjamaar 
\cite{Meinrenken98,Meinrenken-Sjamaar}. Next we explain the functorial properties of the 
``formal geometric quantization'' of non-compact Hamiltonian manifolds \cite{pep-formal}.

\subsection{Quantization commutes with reduction: the compact case}

Let $M$ be a {\em compact} Hamiltonian $K$-manifold with symplectic form $\Omega$ 
and moment map $\Phi_K: M\to \kgot^{*}$ characterized by the relation 
\begin{equation}\label{eq:hamiltonian-action}
    \iota(VX)\Omega= -d\langle\Phi_K,X\rangle,\quad X\in\kgot, 
\end{equation}
where $VX$ is the vector field on $M$ 
generated by $X\in \kgot$.

Let $J$ be a $K$-invariant {\em almost} complex structure on $M$ which is assumed 
to be compatible with the symplectic form : $\Omega(-,J-)$ defines a Riemannian metric on $M$. 
We denote $RR^{^K}(M,-)$ 
the Riemann-Roch character defined by $J$. Let us recall the definition of this map. 

Let $E\to M$ be a complex $K$-vector bundle. The almost complex structure on $M$ 
gives the decomposition $\wedge \T^{*} M \otimes \C =\oplus_{i,j}\wedge^{i,j}\T^* M$
of the bundle of differential forms. Using Hermitian structure in the tangent 
bundle $\T M$ of $M$, and in the fibers of $E$, we define a 
Dolbeault-Dirac operator $\overline{\partial}_E+ 
\overline{\partial}^*_E
:\Acal^{0,even}(M,E)\to\Acal^{0,odd}(M,E)$,
where $\Acal^{i,j}(M,E):=\Gamma(M,\wedge^{i,j}\T^{*}M\otimes_{\C}E)$ 
is the space of $E$-valued forms of type $(i,j)$. The Riemann-Roch
character $RR^{^K}(M,E)$ is defined as the index of the elliptic
operator $\overline{\partial}_E+ \overline{\partial}^*_E$:
$$
RR^{^K}(M,E)= \indice^K_M(\overline{\partial}_E + \overline{\partial}^*_E)
$$
viewed as an element of $R(K)$, the character ring of $K$. 

\medskip

In the Kostant-Souriau framework, a Hamiltonian $K$-manifold
$(M,\Omega,\Phi_K)$ is pre-quantized if there is an equivariant
Hermitian line bundle $\Lcal$ with an invariant Hermitian connection
$\nabla$ such that
\begin{equation}\label{eq:kostant-L}
    L(X)-\nabla_{VX}=i\langle\Phi_K,X\rangle\quad \mathrm{and} \quad
    \nabla^2= -i\Omega,
\end{equation}
for every $X\in\kgot$. Here $L(X)$ is the infinitesimal action of $X\in\kgot$ on the sections 
of $\Lcal\to M$. $(\Lcal,\nabla)$ is also called a Kostant-Souriau line bundle. Remark that 
conditions (\ref{eq:kostant-L}) imply, via the equivariant
Bianchi formula, the relation (\ref{eq:hamiltonian-action}).

We will now recall the notion of geometric quantization.

\begin{defi}\label{def.geom-quant}
When $(M,\Omega,\Phi_K)$ is prequantized by a line bundle $\Lcal$, the geometric quantization of 
$M$ is defined as the index $RR^K(M,\Lcal)$ : we denote it 
$$
\Qcal_K(M,\Omega)\in R(K),
$$
\end{defi}

In order to simplify the notation, we will use also the notation $\Qcal_K(M)$ for the geometric quantization 
of $(M,\Omega,\Phi_K)$.

\begin{rem}\label{rem:euler}
Suppose that $(M,\Omega,J)$ is a {\em compact} K\"ahler manifold pre-quantized by a holomorphic 
line bundle $\Lcal$. Then  

$\bullet$ $\Qcal_K(M,\Omega)$ coincides with 
the holomorphic Euler characteristic of $(M,\Lcal)$,

$\bullet$ for $k\in\N$ large enough, $\Qcal_K(M,k\Omega)\in R(K)$ is equal to the $K$-module formed by the 
holomorphic sections of $\Lcal^{\otimes k}\to M$.
\end{rem}

One wants to compute the $K$-multiplicities of $\Qcal_K(M)$ 
in geometrical terms. A fundamental result of Marsden and Weinstein asserts that if 
$\xi\in\kgot^{*}$ is a regular value of the moment map $\Phi$, the reduced space 
(or symplectic quotient)
$$
M_{\xi}:=\Phi^{-1}_K(\xi)/K_{\xi}
$$ 
is an orbifold equipped with a 
symplectic structure $\Omega_{\xi}$. For any dominant 
weight $\mu\in \what{K}$ which is a regular value of 
$\Phi$, 
$$
\Lcal(\mu):=(\Lcal\vert_{\Phi_K{-1}(\mu)}\otimes \C_{-\mu})/K_\mu
$$
is a Kostant-Souriau line orbibundle over $(M_{\mu},\Omega_{\mu})$. 
The definition of the index of the Dolbeault-Dirac operator carries over
to the orbifold case, hence $\Qcal(M_{\mu})\in \Z$ makes sense as in Definition 
\ref{def.geom-quant}. In
\cite{Meinrenken-Sjamaar}, this is extended further to the case of
singular symplectic quotients, using partial (or shift)
de-singularization. So the integer $\Qcal(M_{\mu})\in \Z$ is well
defined for every $\mu\in \what{K}$ : in particular
$\Qcal(M_{\mu})=0$ if $\mu\notin\Phi_K(M)$.

The following theorem was conjectured by Guillemin-Sternberg
\cite{Guillemin-Sternberg82} and is known as ``quantization commutes
with reduction''
\cite{Meinrenken98,Meinrenken-Sjamaar,Tian-Zhang98,pep-RR}. For
complete references on the subject the reader should consult
\cite{Sjamaar96,Vergne-Bourbaki}.

\begin{theo}[Meinrenken, Meinrenken-Sjamaar] \label{theo:Q-R}
We have the following equality in $R(K)$ :
$$
\Qcal_K(M)=\sum_{\mu\in \what{K}}\Qcal(M_{\mu})\, V_{\mu}^{K}\ .
$$
\end{theo}

\subsection{Formal quantization of non-compact Hamiltonian manifolds}\label{sec.formal.quant}

Suppose now that $M$ is {\em non-compact} but that the moment map
$\Phi_K: M\to\kgot^*$ is assumed to be {\em proper} (we will simply say
``$M$ is proper"). In this situation the geometric
quantization of $M$ as an index of an elliptic operator is not well
defined. Nevertheless the integers $\Qcal(M_{\mu}),\mu\in \what{K}$
are well defined since the symplectic quotients $M_\mu$ are
{\em compact}.

\medskip

A representation $E$ of $K$ is  admissible if it has finite
$K$-multiplicities : \break $\dim(\hom_K(V_\mu^K,E))<\infty$ for
every $\mu\in\what{K}$.  Let $\Rfor(K)$  be the Grothendieck group 
associated to the $K$-admissible representations. We have an inclusion map $R(K)\croc \Rfor(K)$ 
and  $\Rfor(K)$ is canonically identify with $\hom_\Z(R(K),\Z)$. Moreover the tensor product induces an 
$R(K)$-module structure on $\Rfor(K)$ since $E\otimes V$ is an admissible
representation when $V$ and $E$ are, respectively, a finite dimensional
and an admissible representation of $K$.

\medskip

Following \cite{Weitsman,pep-formal}, we introduce the following
\begin{defi}\label{def:formal-quant}
Suppose that $(M,\Omega,\Phi_K)$ is {\em proper} Hamiltonian $K$-manifold prequantized by a line 
bundle $\Lcal$. The formal geometric quantization 
of $(M,\Omega)$ is the element of
$\Rfor(K)$ defined by
$$
\qfor_K(M,\Omega)=\sum_{\mu\in \what{K}}\Qcal(M_{\mu})\, V_{\mu}^{K}\ .
$$
\end{defi}

When the symplectic structure $\Omega$ is understood, we will write $\qfor_K(M)$ for the 
formal geometric quantization of $(M,\Omega,\Phi_K)$.

For a Hamiltonian $K$-manifold $M$ with proper moment map $\Phi_K$, the convexity Theorem 
\cite{Kirwan.84.bis,L-M-T-W} asserts that 
\begin{equation}\label{eq:kirwan-polytope}
\Delta_K(M):= \Phi_K(M)\cap\tgot^*_+
\end{equation}
is a convex rational polyhedron, that one calls the {\em Kirwan polyhedron}.

We will need the following lemma in the next sections.

\begin{lem}\label{lem:Delta-K-M}
Let $(M,\Omega_M)$ and $(N,\Omega_N)$ be two prequantized {\em proper} Hamiltonian $K$-manifold. 
Suppose that $\qfor_K(M,k\Omega_M)=\qfor_K(N,k\Omega_N)$ for any integer $k\geq 1$. Then 
$\Delta_K(M)=\Delta_K(N)$.
\end{lem}
\begin{proof} We check that 
for any $\mu\in\what{K}$ the multiplicity of $V_{k\mu}^K$ in $\qfor_K(M,k\Omega_M)$ is equal to 
$\Qcal(M_\mu,k\Omega_\mu)$. The Atiyah-Singer Riemann-Roch formula gives us the following estimate
$$
\Qcal(M_\mu,k\Omega_\mu)\sim {\rm cst}\  k^r \ {\rm vol}(M_\mu)
$$
when $k$ goes to infinity. Here ${\rm cst}>0$, $r=\dim M_\mu/2$ and ${\rm vol}(M_\mu)$ 
is the symplectic volume of $M_\mu$. Hence, the hypothesis ``$\qfor_K(M,k\Omega_M)=\qfor_K(N,k\Omega_N)$ 
for any integer $k\geq 1$'' implies that $\Phi_K(M)\cap\what{K}=\Phi_K(N)\cap\what{K}$. 

Take an integer $R\geq 1$. By considering the multiplicities of $V_{k\mu}^K$ in  
$\qfor_K(M,k R\,\Omega_M)$, we prove in the same way that  
$\Phi_K(M)\cap\frac{1}{R}\what{K}=\Phi_K(N)\cap\frac{1}{R}\what{K}$. 
Finally we get that 
$$
\Phi_K(M)\cap\big\{\frac{\mu}{R}\ \mid \mu\in\what{K}, R\geq 1\big\}=
\Phi_K(N)\cap\big\{\frac{\mu}{R}\ \mid \mu\in\what{K}, R\geq 1\big\}.
$$
The proof follows since $\big\{\frac{\mu}{R}\ \mid \mu\in\what{K}, R\geq 1\big\}$ is a dense subset of 
the Weyl chamber $\tgot^*_+$.
\end{proof}

\medskip

Let $\varphi: H\to K$ be a morphism between compact connected Lie groups. It induces a pull-back 
morphism $\varphi^*: R(K)\to R(H)$. We want to extend $\varphi^*$ to some elements of 
$\Rfor(K)$. For $\mu\in \what{H}$ and $\lambda\in\what{K}$, let $N^\lambda_\mu$ be the 
multiplicity of $V_\mu^H$ in  $\varphi^*V_\lambda^K$. Formally, the pull-back of 
$E=\sum_{\lambda\in\what{K}}a_\lambda V_\lambda^K$ by $\varphi$ is 
\begin{equation}\label{eq:b-mu}
\varphi^*  E=\sum_{\mu\in\what{H}}b_\mu V_\mu^H\quad {\rm with}\quad 
b_\mu=\sum_{\lambda\in\what{K}}a_\lambda N^\lambda_\mu.
\end{equation} 

\begin{defi}\label{def:admis}
Let $\varphi: H\to K$ be a morphism between compact connected Lie groups.
The element $E=\sum_{\lambda\in\what{K}}a_\lambda V_\lambda^K$ is $H$-admissible if 
for every $\mu\in\what{H}$, the set $\{\lambda\in\what{K}\ \vert a_\lambda N^\lambda_\mu\neq 0\}$ 
is finite. Then the pull-back $\varphi^*E\in \Rfor(H)$ is defined by (\ref{eq:b-mu}).
\end{defi}

The element $\varphi^*E\in \Rfor(H)$ is called the ``restriction'' of $E$ to 
$H$, and will be sometimes  simply denoted by $E\vert_H$.
 
We prove in \cite{pep-formal} the following functorial properties of 
the formal quantization process.

\begin{theo}\label{theo:formal-pep}
\begin{itemize}
\item[\textbf{[P1]}]Let $M_1$ and $M_2$ be respectively
pre-quantized \emph{proper} Hamiltonian $K_1$ and $K_2$-manifolds :
the product $M_1\times M_2$ is then a pre-quantized \emph{proper}
Hamiltonian $K_1\times K_2$-manifold. We have
\begin{equation}\label{eq:functor-produit-formal}
\qfor_{K_1\times K_2}(M_1\times M_2)=\qfor_{K_1}(M_1)\what{\otimes}\,
\qfor_{K_2}(M_2)
\end{equation}
in $\Rfor(K_1\times K_2)\simeq \Rfor(K_1)\what{\otimes} \Rfor(K_2)$.

\item[\textbf{[P2]}] Let $M$ be a pre-quantized \emph{proper} Hamiltonian $K$-manifold. 
Let $\varphi: H\to K$ be a morphism between compact connected Lie groups. Suppose that 
$M$ is still \emph{proper} as a Hamiltonian $H$-manifold. Then $\qfor_K(M)$ is $H$-admissible 
and we have the following equality in $\Rfor(H)$ :
$$
\qfor_K(M)\vert_H=\qfor_H(M).
$$

\item[\textbf{[P3]}] Let $N$ and $M$ be two pre-quantized
Hamiltonian $K$-manifolds where $N$ is \emph{compact} and $M$ is
\emph{proper}. The product $M\times N$ is then proper and we 
have the following equality in $\Rfor(K)$ :
\begin{equation}\label{eq:functor-produit-K}
\qfor_K(M\times N)=\qfor_K(M)\cdot\Qcal_K(N)
\end{equation}
\end{itemize}
\end{theo}

Property $\textbf{[P2]}$ is the hard point in this theorem. In \cite{pep-formal}, we have only 
consider  the case where $\varphi$ is the inclusion of a subgroup. In Appendix 
\ref{sec:appendix-C}, we check the general general case of a morphism $\varphi: H\to K$.

\subsection{Outline of the proof of Theorem \ref{theo:intro}}

We come back to the setting of the introduction. We consider the holomorphic discrete series representation 
$\Qcal_G(\Ocal_\Lambda)$ attached to the Blattner parameter $\Lambda\in \chol$. 
Recall that the coadjoint orbit $\Ocal_\Lambda\simeq G/K_\Lambda$, which is equipped with the 
Kirillov-Kostant-Souriau symplectic form $\Omega_\Lambda$, is pre-quantized by the  line 
bundle $\Lcal_\Lambda:=G\times_{K_\Lambda}\C_\Lambda$.

Consider first the Hamiltonian action of $K$ on $\Ocal_\Lambda$ (here $K$ is a maximal 
compact subgroup of $G$). 
One knows that the corresponding moment map  $\Phi_K:\Ocal_\Lambda\to \kgot^*$ is {\em proper} 
\cite{Duflo-Heckman-Vergne,pep3}. Hence the formal quantization $\qfor_K(\Ocal_\Lambda)$ of 
the $K$-action on $\Ocal_\Lambda$ is well-defined.

Theorem \ref{theo:H-C} tells us that the restriction of the representation $\Qcal_G(\Ocal_\Lambda)$ to $K$ is 
$$
\Qcal_K(\Ocal_\Lambda)=V_\Lambda^K\otimes S(\pgot^+).
$$

Theorem \ref{theo:intro}, restricted to the case where $H=K$, is then equivalent to  
the identity
\begin{equation}\label{eq:H=K}
\qfor_K(\Ocal_\Lambda)=V_\Lambda^K\otimes S(\pgot^+)\quad {\rm in} \quad \Rfor(K),
\end{equation}
that we prove in Section \ref{sec:calcul-Q-Lambda}.

Consider now the situation of a closed connected subgroup $H$ of $K$, such that the restriction 
$\Qcal_H(\Ocal_\Lambda)$ is admissible, e.g. the moment map $\Phi_H:\Ocal_\Lambda\to \hgot^*$ 
is proper (see Proposition \ref{prop:admissible}).  We can apply property \textbf{[P2]} of Theorem 
\ref{theo:formal-pep}. The formal quantization 
$\qfor_H(\Ocal_\Lambda)$ of the $H$-action on $\Ocal_\Lambda$ is equal to the restriction of 
the formal quantization $\qfor_K(\Ocal_\Lambda)$ of the $K$-action on $\Ocal_\Lambda$. Hence
(\ref{eq:H=K}) implies that 
$$
\qfor_H(\Ocal_\Lambda)=\Qcal_H(\Ocal_\Lambda).
$$

So Theorem \ref{theo:intro} is proved for all the admissible restrictions 
$\Qcal_H(\Ocal_\Lambda)$, when one proves it for the case $H=K$.

\section{Computation of $\qfor_K(\Ocal_\Lambda)$}\label{sec:calcul-Q-Lambda}

In this section we prove the following

\begin{theo}\label{theo:formal-lambda}
Let $\Ocal_\Lambda$ be the coadjoint orbit passing through $\Lambda\in\Chol$. We have 
$$
\qfor_K(\Ocal_\Lambda)=V_\Lambda^K\otimes S(\pgot^+).
$$
\end{theo}

Similar computation was done in \cite{pep-ENS} in the setting of a 
geometric quantization of the  ``Spin'' type. 

Note that the formal quantization of  $\Ocal_\Lambda$ behave differently from the ``true'' one, defined 
in Definition \ref{def:quant-lambda}, when $\Lambda\in\Chol\setminus\chol$ : in this case 
$\Qcal_K(\Ocal_\Lambda)=\{0\}$ whereas $\qfor_K(\Ocal_\Lambda)\neq \{0\}$.

The proof of Theorem \ref{theo:formal-lambda} is conducted as follows. 
We introduce in Section \ref{sec:sigma-L} a $K$-transversaly elliptic 
symbol $\sigma_\Lambda$ on $\Ocal_\Lambda$. A direct computation, done in Section \ref{sec:Q-direct}, 
shows that the $K$-equivariant index of $\sigma_\Lambda$ is equal to $V_\Lambda^K\otimes S(\pgot^+)$. 
In Section \ref{sec:deformation}, we use a deformation argument based on the shifting trick to show 
that the index of $\sigma_\Lambda$ coincides with $\qfor_K(\Ocal_\Lambda)$. Putting these results together 
completes the proof of Theorem \ref{theo:formal-lambda}.

\subsection{Transversaly elliptic symbols}

Here we give the basic definitions from the theory of
transversaly elliptic symbols (or operators) 
defined by Atiyah and Singer in \cite{Atiyah74}. For an axiomatic treatment
of the index morphism see Berline-Vergne 
\cite{B-V.inventiones.96.1,B-V.inventiones.96.2} and for a short
introduction see \cite{pep-RR}.

Let $M$ be a {\em compact} $K$-manifold. Let $p:\T M\to M$ 
be the projection, and let $(-,-)_M$ be a $K$-invariant Riemannian metric.
If $E^{0},E^{1}$ are $K$-equivariant complex vector bundles over $M$, a 
$K$-equivariant morphism $\sigma \in \Gamma(\T M,\hom(p^{*}E^{0},p^{*}E^{1}))$  
is called a {\em symbol} on $M$. The subset of all $(m,v)\in \T M$ where 
$\sigma(m,v): E^{0}_{m}\to E^{1}_{m}$ is not invertible is called 
the {\em characteristic set} of $\sigma$, and is denoted by $\Char(\sigma)$. 

In the following, the product of a symbol  $\sigma$ 
by a complex vector bundle $F\to M$, is the symbol
$$
\sigma\otimes F
$$
defined by $\sigma\otimes F(m,v)=\sigma(m,v)\otimes {\rm Id}_{F_m}$ from 
$E^{0}_m\otimes F_m$ to $E^{1}_m\otimes F_m$. Note that $\Char(\sigma\otimes F)=\Char(\sigma)$.

Let $\T_{K}M$ be the following subset of $\T M$ :
$$
   \T_{K}M\ = \left\{(m,v)\in \T M,\ (v,VX(m))_{_{M}}=0 \quad {\rm for\ all}\ 
   X\in\kgot \right\} .
$$

A symbol $\sigma$ is {\em elliptic} if $\sigma$ is 
invertible outside a compact subset of $\T M$ ($\Char(\sigma)$ is
compact), and is {\em transversaly elliptic} if the restriction of $\sigma$ 
to $\T_{K}M$ is invertible outside a compact subset  of $\T_{K}M$ 
($\Char(\sigma)\cap \T_{K}M$ is compact). An elliptic symbol $\sigma$ defines 
an element in the equivariant $K$-theory of $\T M$ with compact 
support, which is denoted by $\K_{K}(\T M)$, and the 
index of $\sigma$ is a virtual finite dimensional representation of $K$, that we denote
$\indice_{M}^K(\sigma)\in R(K)$ 
\cite{Atiyah-Segal68,Atiyah-Singer-1,Atiyah-Singer-2,Atiyah-Singer-3}.

Let 
$$
\Rforc(K)\subset R^{-\infty}(K)
$$ 
be the $R(K)$-submodule formed by all the infinite sum $\sum_{\mu\in\what{K}}m_\mu V_\mu^K$ 
where the map $\mu\in\what{K}\mapsto m_\mu\in\Z$ has at most a {\em
polynomial} growth. The $R(K)$-module $\Rforc(K)$ is the Grothendieck group 
associated to the {\em trace class} virtual $K$-representations: we can associate 
to any $V\in\Rforc(K)$, its trace $k\to \Tr(k,V)$ which is a generalized function on $K$ 
invariant by conjugation. In Section \ref{sec:Q-direct}, we use the fact that the trace 
defines a morphism of $R(K)$-module
\begin{equation}\label{eq:trace}
\Rforc(K)\croc \fgene(K)^K.
\end{equation}

A {\em transversaly elliptic} symbol $\sigma$ defines an element of 
$\K_{K}(\T_{K}M)$, and the index of $\sigma$ is defined as a trace class virtual 
representation of $K$,  that we still denote
$\indice_{M}^K(\sigma)\in \Rforc(K)$. See \cite{Atiyah74} for the analytic index and 
\cite{B-V.inventiones.96.1,B-V.inventiones.96.2} for the cohomological one. 
Remark that any elliptic symbol of $\T M$ is transversaly elliptic, hence 
we have a restriction map $\K_{K}(\T M)\to \K_{K}(\T_{K}M)$, and 
a commutative diagram
\begin{equation}\label{indice.generalise}
\xymatrix{
\K_{K}(\T M)\ar[r]\ar[d]_{\indice_{M}^K} & 
\K_{K}(\T_{K}M)\ar[d]^{\indice_{M}^K}\\
R(K)\ar[r] &  \Rforc(K)\ .
   }
\end{equation} 

\medskip

\medskip

Using the {\em excision property}, one can easily show that the index map 
$\indice_{\Ucal}^{K}:\K_{K}(\T_{K}\Ucal)\to \Rforc(K)$
is still defined when $\Ucal$ is a $K$-invariant relatively compact 
open subset of a $K$-manifold (see \cite{pep-RR}[section 3.1]).

\subsection{The transversaly elliptic symbol $\sigma_\Lambda$}\label{sec:sigma-L}

Let $\Lambda\in\Chol$. Let us first describe the principal symbol  of the Dolbeault-Dirac operator 
$\overline{\partial}_{\Lcal_\Lambda} + \overline{\partial}^*_{\Lcal_\Lambda}$ defined on the coadjoint orbit 
$\Ocal_\Lambda$. The complex vector bundle $(\T^* \Ocal_\Lambda)^{0,1}$ is $G$-equivariantly identified 
with the tangent bundle $\T \Ocal_\Lambda$ equipped with the complex structure $J_\Lambda$.

Let $h$ be the  Hermitian structure on  $\T \Ocal_\Lambda$ defined by : 
$h(v,w)=\Omega_\Lambda(v,J_\Lambda w) -i \Omega_\Lambda(v,w)$ for $v,w\in \T M$. The symbol 
$$
\Thom(\Ocal_\Lambda,J_\Lambda)\in 
\Gamma\left(\Ocal_\Lambda,\hom(p^{*}(\wedge_{\C}^{even} \T \Ocal_\Lambda),\,p^{*}
(\wedge_{\C}^{odd} \T \Ocal_\Lambda))\right)
$$  
at $(m,v)\in \T \Ocal_\Lambda$ is equal to the Clifford map
\begin{equation}\label{eq.thom.complex}
 \Clif_{m}(v)\ :\ \wedge_{\C}^{even} \T_m \Ocal_\Lambda
\longrightarrow \wedge_{\C}^{odd} \T_m \Ocal_\Lambda,
\end{equation}
where $\Clif_{m}(v).w= v\wedge w - c_{h}(v)w$ for $w\in 
\wedge_{\C}^{\bullet} \T_{x}\Ocal_\Lambda$. Here $c_{h}(v):\wedge_{\C}^{\bullet} 
\T_{m}\Ocal_\Lambda\to\wedge^{\bullet -1} \T_{m}\Ocal_\Lambda$ denotes the 
contraction map relative to $h$. Since $\Clif_{m}(v)^2=-\vert v\vert^2 {\rm Id}$, the map  
$\Clif_{m}(v)$ is invertible for all $v\neq 0$. Hence the characteristic set 
of $\Thom(\Ocal_\Lambda,J_\Lambda)$ corresponds to the $0$-section of $\T\Ocal_\Lambda$.

It is a classical fact that the principal symbol  of the Dolbeault-Dirac operator 
$\overline{\partial}_{\Lcal_\Lambda} + \overline{\partial}^*_{\Lcal_\Lambda}$ is equal to\footnote{Here 
we use an identification $\T^*\Ocal_\Lambda\simeq \T\Ocal_\Lambda$ given by an invariant Riemannian metric.}  
\begin{equation}\label{eq:tau-Lambda}
\tau_\Lambda:=\Thom(\Ocal_\Lambda,J_\Lambda)\otimes \Lcal_\Lambda,
\end{equation}
see \cite{Duistermaat96}.  Here also we have $\Char(\tau_\Lambda)=0-{\rm section\ of}\ \T\Ocal_\Lambda$.  
So $\tau_\Lambda$ is not an elliptic symbol 
since the coadjoint orbit $\Ocal_\Lambda$ is non-compact. 

Following \cite{pep-RR,pep-ENS}, we deform $\tau_\Lambda$ in order to define a 
$K$-transversaly elliptic symbol on $\Ocal_\Lambda$. Consider the moment map 
$\Phi_K:\Ocal_\Lambda\to\kgot^*$. With the help of the $K$-invariant scalar product on 
$\kgot^*$ induced by the Killing form on $\ggot$, we define the $K$-invariant function 
$$
\parallel\Phi_K\parallel^2:\Ocal_\Lambda\to\R.
$$
 Let $\Hcal$ be the 
Hamiltonian vector field for $\frac{-1}{2}\parallel\Phi_K\parallel^2$, i.e. 
the contraction of the symplectic form by $\Hcal$ is equal to the $1$-form 
$\frac{-1}{2}d\parallel\Phi_K\parallel^2$.  The vector field $\Hcal$ 
has the following nice description. The scalar product on 
$\kgot^{*}$ gives an identification $\kgot^{*}\simeq\kgot$,  
hence $\Phi_K$ can be consider as a map from $\Ocal_\Lambda$ to $\kgot$. 
We have then
\begin{equation}\label{def:H}
\Hcal_{m}=(V\Phi_K(m))\vert_{m},\quad m\in \Ocal_\Lambda\ ,
\end{equation}
where $V\Phi_K(m)$ is the vector field on $\Ocal_\Lambda$ generated by 
$\Phi_K(m)\in\kgot$.

\begin{defi}\label{def:sigma-Lambda}
Let $\tau_\Lambda$ be the symbol on $\Ocal_\Lambda$ defined in (\ref{eq:tau-Lambda}). The 
{\em symbol $\tau_\Lambda$ pushed by the vector field $\Hcal$} is the 
symbol $\sigma_\Lambda$  defined by the relation
    $$
   \sigma_{\Lambda}(m,v):=\tau_\Lambda(m,v-\Hcal_{m})
   $$
   for any $(m,v)\in\T \Ocal_\Lambda$. 
\end{defi}

The characteristic set of $\sigma_{\Lambda}$ corresponds to
$\{(m,v)\in \T \Ocal_\Lambda,\ v=\Hcal_m\}$, the graph of the vector field
$\Hcal$. Since $\Hcal$ belongs to the set of tangent vectors to the 
$K$-orbits, we have
\begin{eqnarray*}
\Char\left(\sigma_{\Lambda}\right)\cap \T_{K}\Ocal_\Lambda&=&\{(m,0)\in \T \Ocal_\Lambda,\ 
\Hcal_m=0 \}\\
&\cong& \{m\in \Ocal_\Lambda,\ d\parallel\Phi_K\parallel^2_m=0 \} \ .
\end{eqnarray*}
Therefore the symbol $\sigma_{\Lambda}$ is $K$-transversaly elliptic 
if and only if the set  
$\Cr(\parallel\Phi_K\parallel^2)$ of critical points 
of the function $\parallel\Phi_K\parallel^2$ is {\em compact}. 

We have the following result.
\begin{lem}[\cite{Duflo-Heckman-Vergne,pep3}]\label{lem:critical}
The set $\Cr(\parallel\Phi_K\parallel^{2})\subset \Ocal_\Lambda$ is equal to 
the orbit $K\cdot\Lambda$.
\end{lem}

\begin{coro}
The symbol $\sigma_\Lambda$ is $K$-transversaly elliptic.
\end{coro}

\subsection{Computation of $\indice^K(\sigma_\Lambda)$: the direct approach}\label{sec:Q-direct}

The equivariant index of the symbol $\sigma_\Lambda$ can be defined by different manners.

On one hand, since $\Ocal_\Lambda$ can be imbedded $K$-equivariantly in a compact manifold, one can consider
$\indice^K_{\Ocal_\Lambda}(\sigma_{\Lambda})\in\Rforc(K)$.

On the other hand, for any $K$-invariant relatively compact open neighborhood $\Ucal\subset \Ocal_\Lambda$ 
of $\Cr(\parallel\Phi_K\parallel^2)$, the restriction of 
$\sigma_{\Lambda}$ to $\Ucal$ defines a class $\sigma_{\Lambda}\vert_{\Ucal}
 \in \K_{K}(\T_{K}\Ucal)$. Since the index map is well defined on 
$\Ucal$, we can take its index $\indice_{\Ucal}^K(\sigma_{\Lambda}\vert_{\Ucal})$. 
A direct application of the excision property shows  that  
$\indice^K_{\Ocal_\Lambda}(\sigma_{\Lambda})=\indice_{\Ucal}^K(\sigma_{\Lambda}\vert_{\Ucal})$. 
In order to simplify our notation, the index of $\sigma_\Lambda$ is denoted 
$$
\indice^K(\sigma_{\Lambda})\in\Rforc(K).
$$

The aim of this section is the following

\begin{prop}\label{prop-indice-sigma-Lambda} 
Let $\Lambda\in\Chol$. We have 
$$
\indice^K(\sigma_\Lambda)=S(\pgot^+)\otimes V^K_\Lambda \quad {\rm in}\quad \Rforc(K).
$$
\end{prop}
  
The rest of this section is devoted to the computation of $\indice^K(\sigma_{\Lambda})$. 
A similar computation is done in Section 5.2 of \cite{pep-ENS} in the context of a 
``Spin'' quantization.

Let 
\begin{equation}
\Upsilon: \Ocal_\Lambda\longrightarrow \Ocal'_\Lambda:=K\cdot\Lambda\times \pgot
\end{equation}
be the $K$-equivariant  diffeomorphism defined by : $\Upsilon(g\cdot\Lambda)=(k\cdot\Lambda, X)$
where $g=e^X k$, with $k\in K$ and $X\in\pgot$, is the Cartan decomposition of $g\in G$.

The data $(\Omega_\Lambda,J_\Lambda,\Lcal_\Lambda,\Hcal,\sigma_\Lambda)$,  transported to the manifold 
$\Ocal'_\Lambda$ through $\Upsilon$, is denoted 
$(\Omega'_\Lambda,J'_\Lambda,\Lcal'_\Lambda,\Hcal',\sigma'_\Lambda)$. It is easy to check that 
the line bundle $\Lcal'_\Lambda$ is the pull-back of the line bundle 
$K\times_{K_\Lambda}\C_\Lambda\to K\cdot\Lambda$ to $\Ocal'_\Lambda$.

We consider on $\Ocal'_\Lambda$ the following $K$-equivariant data:
\begin{enumerate}
\item The complex structure $J''_\Lambda$ which is the product $J_{K\cdot\Lambda}\times -{\rm ad}(z_o)$. Here 
$J_{K\cdot\Lambda}$ is the restriction of $J_\Lambda$ to the K\"ahler submanifold 
$K\cdot\Lambda\subset G\cdot\Lambda$, and 
${\rm ad}(z_o)$ is the complex structure on $\pgot$ defined in the introduction.

\item The vector field $\Hcal''$ defined by: $\Hcal''_{\xi, X}=-(0,[\xi,X])$ for $\xi\in K\cdot\Lambda$ and 
$X\in\pgot$.
\end{enumerate}

\begin{defi}\label{defi-push}
We consider on $\Ocal'_\Lambda$ the symbols:
\begin{itemize} 
\item $\tau''_\Lambda:=\Thom(\Ocal'_\Lambda,J''_\Lambda)\otimes \Lcal'_\Lambda$,
\item $\sigma''_\Lambda$ which is the symbol $\tau''_\Lambda$ pushed by the vector field $\Hcal''$ 
(see Def. \ref{def:sigma-Lambda}). 
\end{itemize}
\end{defi}

\begin{prop}\label{prop:sigma-sigma}
$\bullet$ The symbol $\sigma''_\Lambda$ is a $K$-transversaly elliptic symbol on $\Ocal'_\Lambda$.

$\bullet$ If $\Ucal$ is a sufficiently small $K$-invariant neighborhood  of $K\cdot\Lambda\times \{0\}$ in 
$\Ocal'_\Lambda$, the restrictions $\sigma'_\Lambda\vert_\Ucal$ and $\sigma''_\Lambda\vert_\Ucal$ 
define the same class in $\K_K(\T_K\Ucal)$.
\end{prop}

\begin{proof} The first point is due to the fact that the vector field $\Hcal''$ is tangent to the $K$-orbits 
in $\Ocal'_\Lambda$. Hence 
\begin{eqnarray*}
\Char(\sigma''_\Lambda)\cap\T_K\Ocal'_\Lambda 
&\simeq &\{(\xi,X)\in \Ocal'_\Lambda\ \vert \ \Hcal''_{\xi, X}=0\}\\
&=& K\cdot\Lambda\times \{0\}.
\end{eqnarray*}
Here we have used that $[\xi,X]=0$ for $\xi\in K\cdot\Lambda$ and 
$X\in\pgot$ if and only if $X=0$.

We will prove the second point by using some homotopy arguments. First we consider 
the family of vector fields $\Hcal_t:= (1-t)\Hcal'+t\Hcal''$, $t\in[0,1]$. 
Let $\sigma_t$ be the symbol $\tau'_\Lambda$ pushed by $\Hcal_t$. On checks easily 
that there exists $c>0$ such that 
\begin{equation}\label{eq:H-prime}
\Hcal'_{\xi,X}= \Hcal''_{\xi,X} + o(\parallel X\parallel^2)
\quad {\rm and}\quad 
\parallel\Hcal''_{\xi,X}\parallel^2\geq c \parallel X\parallel^2
\end{equation} 
holds on $\Ocal'_\Lambda$. With the help of (\ref{eq:H-prime}) it is now easy to prove that 
there exists a $K$-invariant neighborhood $\Vcal$  of $K\cdot\Lambda\times \{0\}$ in 
$\Ocal'_\Lambda$ such that 
$$
\Char(\sigma_t\vert_\Vcal)\cap \T_K\Vcal = K\cdot\Lambda\times \{0\}.
$$
for any $t\in[0,1]$. Hence $\sigma_\Lambda'\vert_\Ucal=\sigma_0\vert_\Ucal$ defines the 
same class than $\sigma_1\vert_\Ucal$ in  $\K_K(\T_K\Ucal)$ for any $K$-invariant 
neighborhood $\Ucal$  of $K\cdot\Lambda\times \{0\}$ that is contained in $\Vcal$.

In order to compare the symbols $\sigma''_\Lambda\vert_\Ucal$ and $\sigma_1\vert_\Ucal$, 
we use a deformation argument similar to the one that we use in the proof of 
Lemma 2.2 in \cite{pep-RR}.

Note first that the complex structures $J'_\Lambda$ and $J''_\Lambda$ are {\em equal} on 
$K\cdot\Lambda\times\{0\}\subset \Ocal'_\Lambda$.  We consider the 
family of equivariant bundle maps  $A_u\in \Gamma(\Ocal'_\Lambda,\End(\T\Ocal'_\Lambda)),\ u\in [0,1]$, 
defined by
$$
A_u:= {\rm Id} - u J'_\Lambda J''_\Lambda.
$$
Since $A_u=(1+u){\rm Id}$ on $K\cdot\Lambda\times\{0\}$, there exists 
a $K$-invariant neighborhood $\Ucal$  of $K\cdot\Lambda\times \{0\}$ (contained in $\Vcal$), 
such that $A_u$ is invertible over $\Ucal$ for any $u\in [0,1]$.

Thus $A_u, u\in [0,1]$ defines over $\Ucal$ a family of bundle isomorphisms :  
$A_0={\rm Id}$ and the map $A_1$ is a bundle complex isomorphism 
$$
\underline{A_1}: (\T \Ucal,J''_\Lambda)\longrightarrow(\T \Ucal,J'_\Lambda).
$$
We extend $\underline{A_1}$ to a complex isomorphism 
$\underline{A_1^{\wedge}}: \wedge_{J''_\Lambda}\T \Ucal \longrightarrow\wedge_{J'_\Lambda}\T \Ucal$.
Then $\underline{A_1^{\wedge}}$ induces an isomorphism between the symbols
$\Thom(\Ucal,J''_\Lambda)$ and  $\underline{A_1}^{*}(\Thom(\Ucal,J'_\Lambda))
\ :\ (x,v)\mapsto \Thom(\Ucal,J'_\Lambda)(x,A_1(x)v)$. In the same way 
$\underline{A_1^{\wedge}}$ induces an isomorphism between the symbols
$\sigma''_\Lambda\vert_\Ucal$ and $\underline{A_1}^{*}(\sigma_1\vert_\Ucal)
\ :\ (x,v)\mapsto \tau'_\Lambda(x,A_1(x)(v-\Hcal''_x))$. One checks easily that 
$\underline{A_u}^{*}(\sigma_1\vert_\Ucal), u\in [0,1]$ is an homotopy of transversaly elliptic symbols.

Finally we have proved that $\sigma''_\Lambda\vert_\Ucal$, $\sigma_1\vert_\Ucal$
and  $\sigma'_\Lambda\vert_\Ucal$ define the same class in \break $\K_K(\T_K\Ucal)$.
\end{proof}

\medskip

Here also, the equivariant index of the transversaly elliptic symbol $\sigma_\Lambda''$ 
can be defined either as the $\indice^K_{\Ocal_\Lambda'}(\sigma_{\Lambda}'')$ taken on 
$\Ocal_\Lambda'$, or as the index $\indice_{\Ucal}^K(\sigma_{\Lambda}''\vert_{\Ucal})$ 
taken on any $K$-invariant relatively compact open neighborhood $\Ucal\subset \Ocal_\Lambda'$ 
of $K\cdot\Lambda\times\{0\}$. We denote simply 
$$
\indice^K(\sigma_{\Lambda}'')\in\Rforc(K).
$$
the equivariant index of $\sigma_\Lambda''$. The second point of Proposition \ref{prop:sigma-sigma} 
shows that 
$\indice_{\Ucal}^K(\sigma_{\Lambda}''\vert_{\Ucal})=\indice_{\Ucal}^K(\sigma_{\Lambda}'\vert_{\Ucal})$. 
Hence we know that 
$$
\indice^K(\sigma_{\Lambda})=\indice^K(\sigma''_\Lambda).
$$

In order to compute $\indice^K(\sigma''_\Lambda)$, we use the induction
morphism 
$$
j_*:\K_{K_\Lambda}(\T_{K_\Lambda}\pgot)\longrightarrow\K_{K}(\T_{K}(\Ocal'_\Lambda))
$$ 
defined by Atiyah in \cite{Atiyah74} (see also \cite{pep-RR}[Section 3]). 
The map $j_*$ enjoys two properties: first, $j_*$ is an isomorphism 
and the $K$-index of $\sigma\in\K_{K}(\T_{K}(\Ocal'_\Lambda))$ 
can be computed via the $K_\Lambda$-index of $j_*^{-1}(\sigma)$.

Let $\sigma : p^*(E^+)\to p^*(E^-)$ be a $K$-transversaly elliptic 
symbol on $\Ocal'_\Lambda$, where $p:\T\Ocal'_\Lambda\to 
\Ocal'_\Lambda$ is the projection, and $E^{+},E^{-}$ are equivariant 
vector bundles over $\Ocal'_\Lambda$. So for any $(\xi, X)\in 
K\cdot\Lambda\times\pgot$, we have a collection of linear maps 
$\sigma(\xi,X;v,Y):E^{+}_{(\xi, X)}\to E^{-}_{(\xi, X)}$ depending on the 
tangent vectors $(v,Y)\in \T_\xi (K\cdot\Lambda)\times\pgot$. The $K_\Lambda$-equivariant 
symbol $j_*^{-1}(\sigma)$ is defined 
by
\begin{equation}\label{eq.def.i.star}
j_*^{-1}(\sigma)(X,Y)=\sigma(\Lambda,X; 0,Y):E^{+}_{(\Lambda,X)}
\longrightarrow  E^{-}_{(\Lambda,X)}\quad {\rm for\ any}\quad (X,Y)\in\T\pgot.
\end{equation}

In the case of the symbol $\sigma''_\Lambda$, the super vector bundle
$E^{+}\oplus E^{-}$ over $\Ocal'_\Lambda$ is 
$\wedge^{\bullet}_{J''_\Lambda}\T\Ocal'_\Lambda \otimes \Lcal'_\Lambda$. 
For any $X\in\pgot$, the super vector space  $E^{+}_{(\Lambda,X)}\oplus E^{-}_{(\Lambda,X)}$ is 
equal to 
$$
\wedge^{\bullet}_\C\pgot^- \otimes\wedge^{\bullet}_\C\kgot/\kgot_\Lambda\otimes\C_\Lambda.
$$

Let $\Thom(\pgot^-)$ be the Thom symbol of the complex vector space $\pgot^-\simeq (\pgot,-\ad(z_o))$. 
Let $\tilde\Lambda$ be the vector field on $\pgot^-$ which is generated by $\Lambda\in\kgot^*\simeq \kgot$. 
Let 
$$
\Thom^{\Lambda}(\pgot^-)
$$ 
be the symbol $\Thom(\pgot^-)$ pushed by the 
vector field $\tilde\Lambda$ (see Definition \ref{defi-push}). 
Since the vector field $\tilde\Lambda$ vanishes only at $0\in \pgot^-$, the symbol 
$\Thom^{\Lambda}(\pgot^-)$ is $K_\Lambda$-transversaly elliptic. We have 
\begin{equation}\label{eq:j-sigma-Lambda}
(j_*)^{-1}(\sigma''_\Lambda)=\Thom^{\Lambda}(\pgot^-)\, 
\otimes\wedge^{\bullet}_\C\kgot/\kgot_\Lambda\otimes\C_\Lambda.
\end{equation}

In (\ref{eq:j-sigma-Lambda}), 
our notation uses the structure of $R(K_\Lambda)$-module for
$\K_{K_\Lambda}(\T_{K_\Lambda}\pgot)$, hence we can multiply 
$\Thom^{\Lambda}(\pgot^-)$ by 
$\wedge^{\bullet}_\C\kgot/\kgot_\Lambda\otimes\C_\Lambda$. 

Let $\fgene(K_\Lambda)^{K_\Lambda}$, $\fgene(K)^K$ be respectively the vector spaces  of generalized 
functions on $K_\Lambda$ and $K$ which are invariant relative to the conjugation action. Let  
\begin{equation}\label{eq.induction}
    \indL : \fgene(K_\Lambda)^{K_\Lambda}\longrightarrow \fgene(K)^{K}\ .   
\end{equation}
be the induction map that is defined as follows : for $\phi\in\fgene(K_\Lambda)^{K_\Lambda}$, 
we have 
$$
\int_{K}\indL(\phi)(k)f(k)dk
=\frac{\Vol(K,dk)}{\Vol(K_\Lambda,dk')}\int_{K_\Lambda}\phi(k')f |_{K_\Lambda}(k')dk',
$$ 
for every $f\in\f(K)^{K}$.

Theorem 4.1 of Atiyah in \cite{Atiyah74} tells us that 
\begin{equation}\label{diagram}
\xymatrix{
K_{K_\Lambda}(\T_{K_\Lambda}\pgot)\ar[r]^{j_{*}}\ar[d]_{\indice^{K_\Lambda}} & 
K_{K}(\T_{K}\Ocal'_\Lambda)\ar[d]^{\indice^K}\\
\fgene(K_\Lambda)^{K_\Lambda}\ar[r]_{\indL} &  \fgene(K)^{K}\ .
   }
\end{equation}
is a commutative diagram\footnote{Here we look at $\Rforc(K_\Lambda)$ and $\Rforc(K)$ as subspaces of 
$\fgene(K_\Lambda)^{K_\Lambda}$ and  $\fgene(K)^{K}$ by using the {\em trace map} (see \ref{eq:trace}).}. 
In other words, $\indice^K(\sigma)=\indL(\indice^{K_\Lambda}(j_*^{-1}(\sigma)))$.
With (\ref{eq:j-sigma-Lambda}) we get 
$$
\indice^K(\sigma''_\Lambda)=\indL\left(\indice^{K_\Lambda}(\Thom^{\Lambda}(\pgot^-)) 
\otimes\wedge^{\bullet}_\C\kgot/\kgot_\Lambda\otimes\C_\Lambda\right).
$$

We know from \cite{pep-RR}[Section 5.1] that the $K_\Lambda$-index of $\Thom^{\Lambda}(\pgot^-)$ is equal 
to the symmetric algebra $S(\pgot^+)$ viewed as a $K_\Lambda$-module.  Since $S(\pgot^+)$ is 
a $K$-module, we have  
\begin{eqnarray*}
\indice^K(\sigma''_\Lambda)&=&\indL\left(S(\pgot^+)\vert_{K_\Lambda}
\otimes\wedge^{\bullet}_\C\kgot/\kgot_\Lambda\otimes\C_\Lambda\right)\\
&=& S(\pgot^+)\otimes\indL\left(\wedge^{\bullet}_\C\kgot/\kgot_\Lambda\otimes\C_\Lambda\right)\\
&=& S(\pgot^+)\otimes V^K_\Lambda.
\end{eqnarray*}

The proof of Proposition \ref{prop-indice-sigma-Lambda} is then completed. See the Appendix in 
\cite{pep-RR} for the relation $\indL\left(\wedge^{\bullet}_\C\kgot/\kgot_\Lambda\otimes\C_\Lambda\right)
=V^K_\Lambda$.

\subsection{Computation of $\indice^K(\sigma_\Lambda)$: the shifting trick}\label{sec:deformation}

This section is devoted to the proof of the following 

\begin{prop}\label{prop:formal-lambda}
Let $\Ocal_\Lambda$ be the coadjoint orbit passing through $\Lambda\in\Chol$. 

For any $\mu\in\what{K}$, the multiplicity of 
$V_\mu^K$ in $\indice^K(\sigma_\Lambda)$ is equal to $\Qcal((\Ocal_\Lambda)_\mu)$. In other words 
we have 
$$
\indice^K(\sigma_\Lambda)=\qfor_K(\Ocal_\Lambda).
$$
\end{prop}

The proof, which follows the same line of Section 4.1 in \cite{pep-ENS}, starts 
with the classical ``shifting trick''. For any $V\in\Rfor(K)$, we denote $[V]^K\in\Z$ 
the multiplicity of the trivial representation in $V$.

By definition the multiplicity $\mm_\Lambda(\mu)$ of $V_\mu^K$ in $\indice^K(\sigma_\Lambda)$ 
is equal to \break 
$[\indice^K(\sigma_\Lambda)\otimes (V_{\mu}^K)^*]^K$,
where $(V_{\mu}^K)^*$ is the (complex) dual of $V_\mu^K$. 
The Borel-Weil Theorem tells us that the representation $V_{\mu}^K$ is equal to the 
$K$-equivariant Riemann-Roch character 
$$
RR^{^K}(K\cdot\mu,\C_{[\mu]}),
$$
where $\C_{[\mu]}\simeq K\times_{K_\mu} \C_\mu$ is the prequantum line bundle over the coadjoint orbit
$K\cdot\mu$. Note that $K\cdot\mu$ is equipped with the K\"ahler structure 
$(\Omega_{K\cdot\mu},J_{K\cdot\mu})$ where $\Omega_{K\cdot\mu}$ is the 
Kirillov-Kostant-Souriau symplectic form, and $J_{K\cdot\mu}$  is the $K$-invariant 
compatible (integrable) complex structure.

Hence the dual $(V_{\mu}^K)^*$ is equal to 
$RR^{^K}(\overline{K\cdot\mu},\C_{[-\mu]})$, 
where $\overline{K\cdot\mu}$ is the coadjoint orbit 
$K\cdot\mu$ equipped with the opposite K\"ahler structure 
$(-\Omega_{K\cdot\mu},-J_{K\cdot\mu})$. Let $\Thom(\overline{K\cdot\mu})$ 
be the equivariant Thom symbol on $(K\cdot\mu,-J_{K\cdot\mu})$. 
Then $(V_{\mu}^K)^*$ is equal to $\indice^K_{K\cdot\mu}(
\Thom(\overline{K\cdot\mu})\otimes \C_{[-\mu]})$. 

Let $\Thom(\Ocal_\Lambda)$ be the Thom symbol on $(\Ocal_\Lambda,J_\Lambda)$. Like in section 
\ref{sec:sigma-L}, let $\Hcal$ be the 
Hamiltonian vector field  of $\frac{-1}{2}\parallel\Phi_K\parallel^2: \Ocal_\Lambda\to\R$. 
We denote by  $\Thom^{\Hcal}(\Ocal_\Lambda)$ the symbol  $\Thom(\Ocal_\Lambda)$ pushed by the 
vector field $\Hcal$ :
$$
\Thom^{\Hcal}(\Ocal_\Lambda)(m,v):=\Thom(\Ocal_\Lambda)(m,v-\Hcal_{m}), \quad
(m,v)\in\T \Ocal_\Lambda.
$$ 

Since $\indice^K(\sigma_\Lambda)$ is equal to $\indice^K_{\Ocal_\Lambda}\left(
\Thom^{\Hcal}(\Ocal_\Lambda)\otimes \Lcal_\Lambda\right)$, the multiplicative property of the index 
\cite{Atiyah74}[Theorem 3.5]  gives 
\begin{equation}\label{eq:m-mu-1}
\mm_\Lambda(\mu)=\left[
\indice^K_{\Ocal_\Lambda\times K\cdot\mu}\left(
(\Thom^{\Hcal}(\Ocal_\Lambda)\otimes \Lcal_\Lambda)
\odot
(\Thom(\overline{K\cdot\mu})\otimes \C_{[-\mu]})
\right)\right]^K\ .
\end{equation}
See \cite{Atiyah74,pep-RR}, for the definition of the exterior product 
$$\odot:\K_K(\T_K \Ocal_\Lambda)\times\K_K(\T(K\cdot\mu))\to 
\K_K(\T_K (\Ocal_\Lambda\times K\cdot\mu)).
$$

It is easy to check that the product $\Thom(\Ocal_\Lambda)\odot
\Thom(\overline{K\cdot\mu})$ is equal to the Thom symbol 
$\Thom(\Ocal_\Lambda\times\overline{K\cdot\mu})$ on the manifold 
$(\Ocal_\Lambda\times\overline{K\cdot\mu}, J_\Lambda\times -J_{K\cdot \mu})$. 
Hence the product $\Thom^{\Hcal}(\Ocal_\Lambda)\odot
\Thom(\overline{K\cdot\mu})$ is equal to the Thom symbol \break 
$\Thom(\Ocal_\Lambda\times\overline{K\cdot\mu})$ pushed by the vector field $(\Hcal, 0)$ : let us denote 
it \break $\Thom^{(\Hcal,0)}(\Ocal_\Lambda\times\overline{K\cdot\mu})$.

The tensor product 
$$
\Lcal:=\Lcal_\Lambda\otimes \C_{[-\mu]}
$$
is a prequantum line bundle over the symplectic manifold $\Ocal_\Lambda\times\overline{K\cdot\mu}$.

Finally (\ref{eq:m-mu-1}) can be rewritten as 
\begin{equation}\label{eq:m-mu-2}
\mm_\Lambda(\mu)=\left[
\indice^K_{\Ocal_\Lambda\times K\cdot\mu}\left(
\Thom^{(\Hcal,0)}(\Ocal_\Lambda\times\overline{K\cdot\mu})\otimes \Lcal \right)\right]^K\ .
\end{equation}

The moment map relative to the Hamiltonian $K$-action on 
$\Ocal_\Lambda\times\overline{K\cdot\mu}$ is 
\begin{eqnarray}\label{Phi-mu}
  \Phi_{1}:\Ocal_\Lambda\times\overline{K\cdot\mu}&\longrightarrow&\kgot^*\nonumber\\
  (m,\xi)&\longmapsto &\Phi_K(m)-\xi 
\end{eqnarray}

For any $t\in\R$, we consider the map $\Phi_{t}:
\Ocal_\Lambda\times\overline{K\cdot\mu}\to \kgot^{*},\ 
 \Phi_{t}(m,\xi):=\Phi(m)-t\,\xi$. Let $\Hcal_t$ be the Hamiltonian vector field 
of $\frac{-1}{2}\parallel\Phi_t\parallel^2: \Ocal_\Lambda\times\overline{K\cdot\mu}\to\R$. 
We denoted $\Thom^{\Hcal_t}(\Ocal_\Lambda\times\overline{K\cdot\mu})$ the symbol 
$\Thom(\Ocal_\Lambda\times\overline{K\cdot\mu})$ pushed by the 
vector field $\Hcal_t$.

We have the fundamental 

\begin{prop}\label{prop:C-t-compact}
$\bullet$ There exists a compact subset $\Kcal$ of $\Ocal_\Lambda$ such that 
$$
\Cr(\parallel\Phi_t\parallel^2)\subset \Kcal\times K\cdot\mu
$$
for any $t\in [0,1]$.

$\bullet$ The symbols $\Thom^{(\Hcal,0)}(\Ocal_\Lambda\times\overline{K\cdot\mu})$ and 
$\Thom^{\Hcal_t}(\Ocal_\Lambda\times\overline{K\cdot\mu}), t\in [0,1]$ are $K$-transversaly elliptic.

$\bullet$ The symbols $\Thom^{(\Hcal,0)}(\Ocal_\Lambda\times\overline{K\cdot\mu})$ and 
$\Thom^{\Hcal_t}(\Ocal_\Lambda\times\overline{K\cdot\mu}), t\in [0,1]$ define the same class in 
$\K_K(\T_K(\Ocal_\Lambda\times K\cdot\mu))$.
\end{prop}

\begin{proof} The proof of the first point is given in \cite{pep-ENS}[Section 5.3.]  
when $\Lambda$ is regular. In the Appendix, we propose another (simpler) proof that 
we learn from Mich\`ele Vergne. 
For the second point we check that 
$$
\Char\left(\Thom^{(\Hcal,0)}(\Ocal_\Lambda\times\overline{K\cdot\mu})\right)\cap 
\T_K\left(\Ocal_\Lambda\times\overline{K\cdot\mu})\right)\simeq
\Cr(\parallel\Phi_K\parallel^2)\times K\cdot\mu\\
$$
and 
$$
\Char\left(\Thom^{\Hcal_t}(\Ocal_\Lambda\times\overline{K\cdot\mu})\right)\cap 
\T_K\left(\Ocal_\Lambda\times\overline{K\cdot\mu})\right)\simeq
\Cr(\parallel\Phi_t\parallel^2)\\
$$
are compact subsets of the $0$-section of $\T(\Ocal_\Lambda\times\overline{K\cdot\mu})$.

Since $\Cr(\parallel\Phi_t\parallel^2)$ stay in the compact set $\Kcal\times K\cdot\mu$ 
for any $t\in[0,1]$, the family $\Thom^{\Hcal_t}(\Ocal_\Lambda\times\overline{K\cdot\mu})$ 
is an homotopy of transversaly elliptic symbol : hence 
they define the same class in 
$\K_K(\T_K(\Ocal_\Lambda\times K\cdot\mu))$.

The vector field $\Hcal_0$ on $\Ocal_\Lambda\times K\cdot\mu$ is equal to $(\Hcal, V)$ where 
$$
V(m,\xi)\in \T_\xi(K\cdot\mu)
$$
for any $ (m,\xi)\in\Ocal_\Lambda\times K\cdot\mu$. We use the deformation 
$(\Hcal, sV), s\in [0,1]$: let $\Thom^{(\Hcal, sV)}(\Ocal_\Lambda\times\overline{K\cdot\mu})$ 
be the Thom symbol pushed by the vector field $(\Hcal, sV)$. It is easy to check that 
there exists a compact subset of $\Kcal'\subset \T(\Ocal_\Lambda\times K\cdot\mu)$ such that 
$$
\Char\left(\Thom^{(\Hcal, sV)}(\Ocal_\Lambda\times\overline{K\cdot\mu})\right)\cap 
\T_K\left(\Ocal_\Lambda\times\overline{K\cdot\mu})\right)\subset \Kcal'
$$
for any $s\in [0,1]$. The family 
$\Thom^{(\Hcal, sV)}(\Ocal_\Lambda\times\overline{K\cdot\mu}), s\in[0,1]$ is then 
an homotopy of transversaly elliptic symbols : hence 
$\Thom^{(\Hcal, 0)}(\Ocal_\Lambda\times\overline{K\cdot\mu})$ and \break 
$\Thom^{\Hcal_0}(\Ocal_\Lambda\times\overline{K\cdot\mu})$ define the same class in 
$\K_K(\T_K(\Ocal_\Lambda\times K\cdot\mu))$. 
\end{proof}

Following the former proposition and (\ref{eq:m-mu-2}), we have 
\begin{equation}\label{eq:m-mu-3}
\mm_\Lambda(\mu)=\left[
\indice^K_{\Ocal_\Lambda\times K\cdot\mu}\left(
\Thom^{\Hcal_1}(\Ocal_\Lambda\times\overline{K\cdot\mu})\otimes \Lcal \right)\right]^K\ .
\end{equation}

We are  now in the following setting : 

$\bullet$ $\Xcal:=\Ocal_\Lambda\times\overline{K\cdot\mu}$ is a Hamiltonian 
$K$-manifold with a {\em proper} moment map $\Phi_1:\Xcal\to\kgot^*$,

$\bullet$ $\Lcal$ is a prequantum line bundle over $\Xcal$,

$\bullet$ the Hamiltonian vector field 
$\Hcal_1$ of the function $\frac{-1}{2}\parallel\Phi_1\parallel^2$ vanishes on a compact subset. 

Hence the ``pushed'' Thom symbol $\Thom^{\Hcal_1}(\Xcal)$ 
is $K$-transversaly elliptic on $\Xcal$. In this context we can consider 
the equivariant index $\indice^K(\Thom^{\Hcal_1}(\Xcal)\otimes\Lcal)$, and 
we have the following theorem 

\begin{theo}[\cite{pep-RR, pep-formal}]\label{theo:loc-index}
The multiplicity of the trivial representation in \break $\indice^K(\Thom^{\Hcal_1}(\Xcal)\otimes\Lcal)$ 
is equal to $\Qcal(\Xcal_0)$, where $\Xcal_0$ is the (compact) symplectic reduction of $\Xcal$ at $0$.
\end{theo}

If we apply Theorem \ref{theo:loc-index} to $\Xcal=\Ocal_\Lambda\times\overline{K\cdot\mu}$ we have 
$\Xcal_0\simeq (\Ocal_\Lambda)_\mu$, and then we can conclude that
$$
\mm_\Lambda(\mu)=\Qcal\left((\Ocal_\Lambda)_\mu\right).
$$
The proof of Proposition \ref{prop:formal-lambda} is then completed.

\section{Properness and admissibility}\label{sec:proper}

In this section, we consider an element $\Lambda\in\Chol$, and a closed connected subgroup $H$ of 
$K$. We consider the representation $V_\Lambda^K\otimes S(\pgot^+)$ of $K$: note that 
it is an admissible $K$-representation since the circle group $Z(K)$ acts on $\pgot^+$ 
by multiplication. We are interested in the condition 
$$ 
{\bf C1}\qquad \mathrm{The\ representation\ } V_\Lambda^K\otimes S(\pgot^+)\vert_H\ 
\mathrm {is\ an\ admissible}\   
H-\mathrm{representation}. 
$$

Let $\Phi_H:\Ocal_\Lambda\to\hgot^*$ be the moment map relative to the Hamiltonian action of $H$ 
on the coadjoint orbit $\Ocal_\Lambda:=G\cdot\Lambda$:  the map $\Phi_H$ is simply the composition of the 
moment map $\Phi_K:\Ocal_\Lambda\to\kgot^*$ with the canonical projection $\pi:\kgot^*\to\hgot^*$.  

Let us consider the condition
$$
{\bf C2}\qquad \mathrm{The\ map\ } \Phi_H:\Ocal_\Lambda\to\hgot^*\ \mathrm {is\ proper}. 
$$
The aim of this section is to prove that ${\bf C1}\Longleftrightarrow {\bf C2}$. During the proof,  
we will obtain other equivalent conditions. 

We start with the 

\begin{lem}\label{lem:C2impliqueC1}
We have ${\bf C2}\Longrightarrow{\bf C1}$.
\end{lem}
\begin{proof}
We have proved in Section \ref{sec:calcul-Q-Lambda} that $V_\Lambda^K\otimes S(\pgot^+)$ is equal to 
$\qfor_K(\Ocal_\Lambda)$. Then, Property {\bf P2} of Theorem \ref{theo:formal-pep} tells us that 
the properness of $\Phi_H$ implies the $H$-admissibility of $V_\Lambda^K\otimes S(\pgot^+)\vert_H$. 
Since this fact is easy to prove, let's recall it.

For $\mu\in \what{K}$ and $\nu\in \what{H}$ we denote
$N^\mu_{\nu}=\dim (\hom_H(V_\nu^H,V_\mu^K|_H))$ the multiplicity of
$V_\nu^H$ in the restriction $V_\mu^K|_H$. Since 
$V_\Lambda^K\otimes S(\pgot^+)=\sum_{\mu\in \what{K}}\Qcal(M_{\mu})\, V_{\mu}^{K}$ 
we know that the multiplicity (possibly infinite) of $V_\nu^H$ in $V_\Lambda^K\otimes S(\pgot^+)$ is 
\begin{equation}\label{eq:somme-mu}
\sum_{\mu\in \what{K}}N^\mu_{\nu}\Qcal\left(M_{\mu}\right).
\end{equation}
Let us see that the former sum is always {\em finite} when {\bf C2} holds. 
Since $V_\mu^K$ is equal to the $K$-quantization of the coadjoint orbit $K\!\cdot\!\mu$, the 
restriction  $V_\mu^K\vert_H$ is equal to the quantization of $K\!\cdot\!\mu$, viewed as a 
Hamiltonian $H$-manifold: the corresponding 
moment map $K\!\cdot\!\mu\to\hgot^*$ is the restriction of the projection $\pi$ to $K\!\cdot\!\mu$.
The ``quantization commutes with reduction'' theorem tells us that the multiplicity $N^\mu_{\nu}$ 
is equal to the quantization of the symplectic reduction of the Hamiltonian $H$-manifold $
K\!\cdot\!\mu$ at $\nu$. In particular
$N^\mu_{\nu}\neq 0$ implies that $ \nu\in \pi(K\cdot\mu)$. Finally
$$
N^\mu_{\nu} \Qcal\left(M_{\mu}\right)\neq 0 \ \Longrightarrow\
\mu\in K\!\cdot \pi^{-1}(\nu)\quad \mathrm{and} \quad
\Phi^{-1}_K(\mu)\neq\emptyset.
$$
These two conditions imply that
we can restrict the sum of (\ref{eq:somme-mu}) to
$$
    \mu\in \what{K}\cap \Phi_K\left( K\cdot\Phi^{-1}_H(\nu)\right)
$$
which is \emph{finite} since $\Phi_H$ is proper.
\end{proof}

\medskip

The rest of this section is dedicated to the proof of ${\bf C1}\Longrightarrow {\bf C2}$. 
Since $V^K_\mu$ is finite dimensional, one notices that ${\bf C1}$ is equivalent to :
$$ 
{\bf C1'}\qquad \mathrm{The\ representation\ } S(\pgot^+)\vert_H\ \mathrm {is\ an\ admissible}\   
H-\mathrm{representation}. 
$$

\subsection{Formal quantization of $\pgot$}

Let us denoted $\pgot^-$, the real vector space $\pgot$ equipped with the  complex 
structure $-\ad(z_o)$ (see the introduction). Let $\Omega_{\pgot}$ be the (constant) symplectic 
structure on $\pgot$ defined by 
\begin{equation}\label{eq:omega-p}
\Omega_{\pgot}(X,Y)= B_\ggot( X,\ad(z_o)Y)
\end{equation}
where $B_\ggot$ is the Killing form on $\ggot$. 

Let $\h$ be the Hermitian structure on $\pgot^-$ defined by $\h(X,Y)=B(X,Y)-i\Omega_{\pgot}(X,Y)$. 
Let ${\rm U}:={\rm U}(\pgot^-)$ be the unitary group with Lie algebra $\ugot$.
We use the isomorphism $\epsilon:\ugot\to\ugot^*$ defined by
$\langle\epsilon(A),B\rangle=-\Tr_\C(AB)\in \R$. For $X,Y\in \pgot$, let
$X\otimes Y^*: \pgot^-\to \pgot^-$ be the $\C$-linear map $Z\mapsto \h(Z,X)Y$.

The action of ${\rm U}$ on $(\pgot,\Omega_{\pgot})$ is Hamiltonian with moment map 
$\Phi_{{\rm U}}: \pgot\to\ugot^*$ defined by
$$
\langle\Phi_{{\rm U}}(X),A\rangle=\Omega_{\pgot}(A(X),X),\quad X\in\pgot,\ A\in\ugot.
$$ 
Via $\epsilon$, the moment map $\Phi_\U$ is defined by
\begin{equation}\label{eq:Phi-p-U}
    \Phi_{{\rm U}}(X)= \frac{1}{i}X\otimes X^*,\quad X\in\pgot.
\end{equation}

The Hamiltonian space $(\pgot,\Omega_{\pgot},\Phi_{{\rm U}})$ is prequantized 
by the trivial line bundle, equipped with the Hermitian structure $\langle s,s'\rangle\vert_X=
e^{\frac{-\|X\|^2}{2}}s\overline{s'}$ and the Hermitian connexion
$\nabla=d- i\theta$ where $\theta$ is the $1$-form on $\pgot$ defined by
$\theta= \Omega_\pgot(X,dX)$.

One sees that $\Phi_{{\rm U}}$ is a proper map. Hence we can consider the formal quantization
$\qfor_\U(\pgot,\Omega_\pgot)\in\Rfor(\U)$ of the $\U$-action on the symplectic manifold 
$(\pgot,\Omega_\pgot)$. We are also interested in $\qfor_\U(\pgot,k\Omega_\pgot)\in\Rfor(\U)$, for 
any integer $k\geq 1$.

\medskip

\begin{lem}[\cite{pep-formal}]\label{lem:Q-formal-p-U}
 The symmetric space $S(\pgot^+)$ is an admissible $U$-representation.  The following equality 
\begin{equation}\label{eq:Q-formal-p-U}
    \qfor_\U(\pgot,k\Omega_\pgot)= S(\pgot^+)
\end{equation}
holds in $\Rfor(U)$, for any $k\geq 1$.
\end{lem}

\begin{proof}
In \cite{pep-formal}, we consider the case $k=1$. The other cases follow since 
the symplectic vector space $(\pgot,k\Omega_\pgot)$ is equivariantly symplectomorphic to 
$(\pgot,\Omega_\pgot)$.
\end{proof}

\subsection{Formal quantization of $\pgot$ relative to the $K$-action}

The adjoint action of $K$ on $\pgot$ defines a morphism $\varphi: K\to \U$. Let us denoted 
by $\varphi:\kgot\to\ugot$ the corresponding morphism of Lie algebra, and by 
$\varphi^*:\ugot^*\to\kgot^*$ the dual linear map. The moment map $\Phi_K:\pgot\to\kgot^*$ 
is equal to the composition of 
$\Phi_\U:\pgot\to\ugot^*$ with $\varphi^*$. Via the identification $\kgot^*\simeq\kgot$ 
given\footnote{The map $\xi\in\kgot^*\mapsto\tilde{\xi}\in\kgot$ is defined by 
the relation $\langle\xi,X\rangle:=-B_\ggot(\tilde{\xi},X)$, $\forall X\in\kgot$.}
 by the Killing form $B_\ggot$, the moment map $\Phi_K$ is defined by
\begin{equation}\label{eq:Phi-p-K}
    \Phi_K(X)=- [X,[z_o,X]]\ \in\ \kgot,\quad X\in\pgot.
\end{equation}
We note that $\langle\Phi_K(X),z_o\rangle= \parallel [z_o,X] \parallel^2>0$ if $X\neq 0$. Hence the moment map
$\Phi_K:\pgot\to\kgot^*$ is {\em proper}.  We use property\textbf{[P2]}
of Theorem \ref{theo:formal-pep} (see also Appendix C) to get from Lemma \ref{lem:Q-formal-p-U} the 

\begin{coro}\label{coro:Q-formal-p-K}
 The symmetric space $S(\pgot^+)$ is an admissible $K$-representation. The following equality 
\begin{equation}\label{eq:Q-formal-p-K}
    \qfor_K(\pgot,k\Omega_\pgot)= S(\pgot^+)
\end{equation}
holds in $\Rfor(K)$, for any $k\geq 1$.
\end{coro}

\medskip

We look now at the Hamiltonian action of a closed connected subgroup $H\subset K$ on $(\pgot,\Omega_\pgot)$. 
The moment map $\Phi_H:\pgot\to\hgot^*$ is  the composition of the 
map $\Phi_K:\pgot\to\kgot^*$ with the canonical projection $\pi:\kgot^*\to\hgot^*$.  In 
this setting, we know from property \textbf{[P2]} that the properness of $\Phi_H$ implies 
that $S(\pgot^+)\vert_H$ is an admissible representation 
of $H$. In \cite{pep-formal}[Section 5], we have proved the converse. Let $\Delta_K(\pgot)$ 
be the Kirwan polyhedral convex set associated to the Hamiltonian action of $K$ on 
$(\pgot,\Omega_\pgot)$.  Let $\hgot^\perp:=\ker(\pi)\subset \kgot^*$. We have the

\begin{lem}[\cite{pep-formal}]\label{lem:closed-orbit-p}
    The following conditions are equivalent :
    \begin{enumerate}
      \item $\Delta_K(\pgot)\cap K\cdot\hgot^\perp=\{0\}$,
      \item the map $\Phi_H:\pgot\to\hgot^*$ is proper,
      \item  The subalgebra $S(\pgot^+)^H$ formed by the $H$-invariant elements is reduced to the constants.
      \item ${\bf C1'} \ :\ \  S(\pgot^+)\vert_H$ is an admissible representation of $H$.
      
    \end{enumerate}
\end{lem}

\begin{proof}
Since the map $\Phi_H:\pgot\to\hgot^*$ is quadratic, the map $\Phi_H$ is proper 
if and only if $\Phi_H^{-1}(0)=\{0\}$. Now it is easy to check that 
$\Phi_K\left(\Phi_H^{-1}(0)\right)=K\cdot\Delta_K(\pgot)\cap \hgot^\perp$. 
Hence $(1)\Longleftrightarrow\Phi_H^{-1}(0)=\{0\}\Longleftrightarrow (2)$.

The equivalence of the last three points uses property \textbf{[P2]} and some basic results 
of Geometric Invariant Theory
(see Lemma 5.2 in \cite{pep-formal}).
\end{proof}

\subsection{Proof of ${\bf C1}\Longrightarrow {\bf C2}$}

Let $\Delta_K(\Ocal_\Lambda)$ be the Kirwan polyhedral convex set 
associated to the Hamiltonian action of $K$ on $(\Ocal_\Lambda,\Omega_\Lambda)$.  

To any non-empty subset $C$ of a real vector space $E$, we associate its asymptotic cone 
$\As(C)\subset E$ formed by all the limits $y=\lim_{k\to\infty} t_k y_k$ where $(t_k)$ is a sequence of 
non-negative reals converging to $0$ and $y_k\in C$. Recall the following basic facts:
\begin{enumerate}
\item $\As(C)$ is a closed cone,
\item $\As(C)=\{0\}$ if $C$ is bounded,
\item $\As(C)= C$ if $C$ is  a closed cone,
\item If a compact Lie group $K$ acts linearly on $E$, then $\As(K\cdot C)=K\cdot\As(C)$.
\end{enumerate}

\begin{prop}\label{prop:Delta-K-Ocal-Lambda}
Let $\Lambda\in\Chol$. We have 
$$
\Delta_K(\Ocal_\Lambda)=\Delta_K(K\cdot\Lambda\times \pgot)
\quad
{\rm and }
\quad
\As\left(\Delta_K(\Ocal_\Lambda)\right)=\Delta_K(\pgot).
$$
\end{prop}

\begin{proof} For any integer $k\geq 1$, the coadjoint orbit $\Ocal_\Lambda$, equipped with the 
symplectic form $k\Omega_{\Lambda}$, is symplectomorphic to $(\Ocal_{k\Lambda},\Omega_{k\Lambda})$. 
Theorem \ref{theo:formal-lambda} shows then that 
$$
\qfor_K(\Ocal_\Lambda, k\Omega_\Lambda)=V_{k\Lambda}^K\otimes S(\pgot^+).
$$

Consider now the product $\Ocal_\Lambda':= K\cdot\Lambda\times\pgot$ equipped with the symplectic 
structure $\Omega_\Lambda':=\Omega_{K\cdot\Lambda}\times\Omega_\pgot$: here 
$\Omega_{K\cdot\Lambda}$ is 
the Kirillov-Kostant-Souriau 
symplectic form and $\Omega_\pgot$ is defined in (\ref{eq:omega-p}). For any integer $k\geq 1$, 
the symplectic manifold is $(\Ocal_\Lambda', k\Omega_\Lambda')$ is pre-quantized by 
$(\Lcal'_\Lambda)^{\otimes k}$, where $\Lcal'_\Lambda$  is the pull-back of the line bundle 
$K\times_{K_\Lambda}\C_\Lambda\to K\cdot\Lambda$ to $\Ocal'_\Lambda$. 

Since $(\pgot,\Omega_\pgot)$ has a {\em proper} $K$-moment map, we can use property \textbf{[P3]}
of Theorem \ref{theo:formal-pep}. We have 
\begin{eqnarray*}
\qfor_K(\Ocal_\Lambda', k\Omega_\Lambda')&=& \Qcal_K(K\cdot\Lambda,k \Omega_{K\cdot\mu})\otimes
\qfor_K(\pgot, k\Omega_\pgot)\\
&=&
V_{k\Lambda}^K\otimes S(\pgot^+).
\end{eqnarray*}

We are now in the setting of Lemma \ref{lem:Delta-K-M} : $(\Ocal_\Lambda, \Omega_\Lambda)$ and 
$(\Ocal_\Lambda', \Omega_\Lambda')$ are two prequantized {\em proper} Hamiltonian $K$-manifold such 
that $\qfor_K(\Ocal_\Lambda, k\Omega_\Lambda)=$ \break 
$\qfor_K(\Ocal_\Lambda', k\Omega_\Lambda')$ for any integer 
$k\geq 1$. This implies that
$\Delta_K(\Ocal_\Lambda)=\Delta_K(\Ocal_\Lambda')$. Hence the first point is proved.

For the other point, we first observe that $\Lambda+\Delta_K(\pgot)\subset \Delta_K(\Ocal_\Lambda')$, so 
$$
\Delta_K(\pgot)=\As\left(\Lambda+\Delta_K(\pgot)\right)\subset \As\left(\Delta_K(\Ocal_\Lambda')\right).
$$

Let $y\in \As\left(\Delta_K(\Ocal_\Lambda')\right)$. We have $y=\lim_{k\to\infty} t_k y_k$ with 
$y_k=y_k'+y_k''$, where $y_k'\in K\cdot\Lambda$, $y_k''\in\Phi_K(\pgot)$, $y_k'+y_k''\in\tgot^*_+$ 
and $t_k$ is a sequence of positive number converging to $0$. Since $y_k'$ is bounded, we have
$$
y=\lim_{k\to\infty} t_k y_k''\in \Phi_K(\pgot)\cap\tgot^*_+.
$$
So we have proved that $y\in \Delta_K(\pgot)$. With the first point we can conclude that
$$
\Delta_K(\pgot)=\As\left(\Delta_K(\Ocal_\Lambda')\right)=\As\left(\Delta_K(\Ocal_\Lambda)\right).
$$ 

\end{proof}

\begin{rem}When $\Lambda\in \Chol$ is $K$-invariant, the K\"ahler manifold $\Ocal_\Lambda$ is 
exactly the Hermitian symmetric space $G/K$. In this situation, McDuff \cite{McDuff} has shown that $G/K$ is 
{\em symplectomorphic} to the symplectic vector space $(\pgot, \Omega_\pgot)$. 

In the light of Proposition \ref{prop:Delta-K-Ocal-Lambda}, we conjecture that 
for any $\Lambda\in \Chol$, the coadjoint orbit $\Ocal_\Lambda$ is $K$-equivariantly 
{\em symplectomorphic} to the product 
$K\cdot\Lambda\times\pgot$ equipped with the symplectic 
structure $\Omega_{K\cdot\Lambda}\times\Omega_\pgot$.
\end{rem}

\medskip

We need the following basic 
\begin{lem}\label{lem:propre}
Let $(M,\Omega)$ be a Hamiltonian $K$-manifold with a {\em proper} moment map $\Phi_K:M\to \kgot^*$. 
Let $H\subset K$ be a closed connected subgroup. Let  $\Phi_H:M\to \kgot^*$ be the moment map 
relative to the action of $H$ on $M$. Suppose that we have 
$$
\As\left(\Delta_K(M)\right)\cap K\cdot \hgot^\perp=\{0\}.
$$

Then there exists $c>0$ such that $\| \Phi_H(m)\|\geq c \|\Phi_K(m)\|$ 
holds outside a compact subset of $M$. In particular $\Phi_H$ is a {\em proper} map.
\end{lem}

\begin{proof}
Suppose that there exists a sequence $m_i\in M$ such that 
$$
\lim_{i\to\infty}\|\Phi_K(m_i)\|=\infty\quad {\rm and} \quad 
\lim_{i\to\infty}\frac{\|\Phi_H(m_i)\|}{\|\Phi_K(m_i)\|}=0.
$$

We put $\Phi_K(m_i)=k_i\cdot y_i$ with $k_i\in K$ and $y_i\in \Delta_K(M)$. 
We have then
$$
\lim_{i\to\infty}\pi\left(k_i\cdot\frac{y_i}{\|y_i\|}\right)=0
$$
where $\pi:\kgot^*\to\hgot^*$ is the projection. 
Here we can assume that the sequence $k_i$ converge to $k\in K$, and that the sequence 
$\frac{y_i}{\|y_i\|}$ converge to $y\in \As(\Delta_K(M))$, $\|y\|=1$. We get then that
$\pi(k\cdot y)=0$. In other words, $y$ is a non-zero element in 
$\As\left(\Delta_K(M)\right)\cap K\cdot\ker(\pi).$
\end{proof}

We can now finish the proof of ${\bf C1}\Longrightarrow {\bf C2}$. We have already check in 
Lemma \ref{lem:closed-orbit-p} that 
$$
{\bf C1}\Longleftrightarrow {\bf C1'}\Longleftrightarrow \Delta_K(\pgot)\cap K\cdot\hgot^\perp=\{0\}.
$$

We have proved in Proposition \ref{prop:Delta-K-Ocal-Lambda} that 
$\Delta_K(\pgot)= \As\left(\Delta_K(\Ocal_\Lambda)\right)$, so 
condition {\bf C1} is equivalent to 
\begin{equation}\label{eq:Delta-M-ker-pi-bis}
\As\left(\Delta_K(\Ocal_\Lambda)\right)\cap K\cdot\hgot^\perp=\{0\}.
\end{equation}

Finally, we know after Lemma \ref{lem:propre} that  (\ref{eq:Delta-M-ker-pi-bis}) implies the 
properness of the moment map $\Phi_H:\Ocal_\Lambda\to\hgot^*$.

\section{Description of $\Delta_K(\pgot)$}\label{sec:delta-K-p}

The purpose of this section is the description of the Kirwan polyhedral cone 
$\Delta_K(\pgot)$ which is attached to the Hamiltonian action of $K$ on 
$(\pgot,\Omega_\pgot)$.

For any root $\alpha\in\Rgot=\Rgot(\ggot_\C,\tgot_\C)$ the corresponding root space $\ggot_\alpha\subset\ggot_\C$
is defined as $\{X\in \ggot_\C\ \vert\ [H,X]=i\langle \alpha, H\rangle X,\ \forall \ H\in \tgot\}$.

For the rest of this section, we work with the system of positive roots
$\Rgot^+_{\rm hol}=\Rgot^+_c\cup\Rgot^{+,z_o}_n$   defined in the introduction.  
For any positive non-compact root $\beta\in \Rgot_n^{+,z_o}$, there are $H_\beta\in\tgot, E_\beta\in\ggot_\beta,
E_{-\beta}\in\ggot_{-\beta}$ such that
\begin{eqnarray}\label{eq:sl2}
[E_{\beta},E_{-\beta}] &=&      i H_\beta\nonumber\\
\overline{E_{\beta}}     &=&      E_{-\beta}\\
B_\ggot(E_{\beta},E_{-\beta})  &=&  \frac{2}{\|\beta\|^2}\nonumber.
\end{eqnarray}
Here $X\mapsto \overline{X}$ is the conjugation on $\ggot_\C$ relative to the real form $\ggot$, and 
the norm $\|-\|^2$ on $\tgot^*$ is induced by the Killing form $B_\ggot$.

Note that conditions (\ref{eq:sl2}) implies that $[i H_\beta,E_{\beta}]= 2 E_{\beta}$,  
$[i H_\beta,E_{-\beta}]= -2 E_{-\beta}$  and 
\begin{equation}\label{eq:H-gamma}
H_\beta\simeq -2\frac{\beta}{\|\beta\|^2}
\end{equation} 
through the isomorphism $\tgot\simeq\tgot^*$.  In particular $iH_\beta,E_{\beta}$ and $E_{-\beta}$ span 
a subalgebra of $\ggot_\C$ isomorphic to $\slgot(2,\C)$.

For $\beta\in \Rgot^{+,z_o}_n$, let $X_\beta=\frac{1}{2}( E_{\beta}+E_{-\beta})$ and
$Y_\beta =\frac{1}{2i}( E_{\beta}-E_{-\beta}) $. 
Thus the set $\{X_\beta,Y_\beta\}_{\beta\in \Rgot^{+,z_o}_n}$ is a real basis of $\pgot$. 
Since $\langle\beta, z_o\rangle=1$ for any $\beta\in \Rgot^{+,z_o}_n$, we have $\ad(z_o)X_\beta  = -Y_\beta$ 
and $\ad(z_o)Y_\beta  = X_\beta$.

We will now describe the restricted root system of $G/K$. Two roots $\alpha,\beta\in\Rgot$ are {\em 
strongly orthogonal}, written $\alpha\spp\beta$, if neither of $\alpha\pm\beta$ is a root. 
One can easily check that strong orthogonality implies orthogonality with respect to the 
scalar product on $\tgot^*$.  

Consider the ``cascade construction''
\begin{eqnarray}\label{eq:cascade}
\Psi&=&\{\gamma_1,\ldots,\gamma_r\},\  {\rm maximal\ set\ constructed\ by\ }:\nonumber\\
&& \gamma_1\ {\rm is \ the\ maximal \ root\  in\ } \Rgot^{+,z_o}_n\nonumber\\
&& \gamma_{i+1}\ {\rm is \ the\ maximal \ root\  in\ } 
\{\beta\in \Rgot^{+,z_o}_n\ \vert\ \beta\spp\gamma_k\ {\rm for}\ k=1,\ldots,i\}.\nonumber
\end{eqnarray}

For the roots $\gamma_k$, we denote simply $X_k,Y_k, H_k$ the elements $X_{\gamma_k}, Y_{\gamma_k}, 
H_{\gamma_k}$. We have the classical  result (see \cite{Helgason}[Prop. 7.4])

\begin{lem}\label{lem:a}
The subspace
$$
\agot:=\sum_{k=1}^r \R X_k
$$
is maximal abelian in $\pgot$.
\end{lem}

Since $\pgot=K\cdot\agot$, it is sufficient to understand the image of 
$\agot$ by $\Phi_K$ to compute $\Delta_K(\pgot)$ :  in fact this Kirwan cone will be computed by 
describing  the image by $\Phi_K$ of a closed cone $\agot_+\subset \agot$, which is a fundamental 
domain for the $K$-action on $\pgot$.

For $\lambda\in\agot^*$, we write 
$$
\ggot^\lambda:=\{X\in \ggot\ \vert\ [H,X]=\langle \lambda, H\rangle X\ {\rm for\ all}\ H\in\agot\}.
$$
If $\ggot^\lambda\neq 0$ and $\lambda\neq 0$, we call $\lambda$ a {\em restricted root} of $\ggot$. The set 
of restrited roots is denoted $\Sigma$. Let $W_\Sigma$ be the group generated 
by the orthogonal symmetries along the hyperplane $\ker(\lambda)$, $\lambda\in \Sigma$. A proof of 
the following classic result can be found in \cite{Knapp-beyond}[Sec. VI.5].
\begin{prop}\label{prop:restricted-root}
$\bullet$ $\Sigma$ is an abstract root system on $\agot^*$.

$\bullet$ The group $W_\Sigma$ is finite and is canonically identify with the quotient $N_K(\agot)/Z_K(\agot)$, 
where $N_K(\agot)$ is the normalizer subgroup of $\agot$ in $K$ and 
$Z_K(\agot)$ is the centralizer subgroup of $\agot$ in $K$.
\end{prop}

With the help of a system  of positive roots $\Sigma^+$, we define the closed chamber 
$$
\agot_+:=\left\{H\in\agot\ \vert \ \langle \lambda, H\rangle\geq 0 \ 
{\rm for\ all}\ \lambda\in\Sigma^+\right\}.
$$
Proposition \ref{prop:restricted-root} tell us then that any $K$-orbit in $\pgot$ intersects $\agot_+$ in a unique 
point.

We have the fundamental 

\begin{prop}\label{prop:a-plus}
For a particular system of positive roots $\Sigma^+$, we have  
$$
\agot_+=\sum_{k=1}^r \R^{\geq 0}(X_1+\cdots+ X_k).
$$
\end{prop}

\begin{proof}The proof is done in Appendix B.
\end{proof}

An element $X$ of the chamber $\sum_{k=1}^r \R^{\geq 0}(X_1+\cdots+ X_k)$ is of the form 
$X=\sum_{k=1}^r t_k X_k$ with $t_1\geq \cdots\geq t_r\geq 0$. Then 
$\Phi_K(X)$, view as an element of $\kgot$, is equal to
\begin{eqnarray*}
\Phi_K(X)&=& -[X,[z_o,X]]\\
&=& \sum_{k,l} t_kt_l [X_k,Y_l]\\
&=& -\frac{1}{2}\sum_{k=1}^r (t_k)^2 H_k.
\end{eqnarray*}
Here we have used the fact $[X_k,Y_l]=0$ for $k\neq l$ since 
$[\ggot_{\gamma_k},\ggot_{\pm\gamma_l}]=0$. When $k=l$, 
one sees that $[X_k,Y_k]= \frac{i}{2}[E_{\gamma_k},E_{-\gamma_k}]= -\frac{1}{2}H_k$.

Since the vector $-\frac{1}{2}H_k\in\tgot$ corresponds to $\frac{\gamma_k}{\|\gamma_k\|^2}$ 
through the identification $\tgot\simeq\tgot^*$. We conclude that 
$$
\Phi_K(X)= \sum_{k=1}^r (t_k)^2 \frac{\gamma_k}{\|\gamma_k\|^2}\in\tgot^*
$$
for $X=\sum_{k=1}^r t_k X_k$.

Let $\tgot^*_+\subset\tgot^*$ be the Weyl chamber defined the by the system of positive compact roots $\Rgot^+_c$. 
Let $\overline{\Chol}\subset\tgot^*_+$ be the Weyl chamber defined by the system of positive roots $\Rgot^+_{\rm hol}$. 
The following proposition will be proved in Appendix B.

\begin{prop}\label{prop:appendixB-2}
All the roots $\gamma_k$ have the same lenghts, and we have 
$$
\overline{\Chol}\cap {\rm Vect}(\gamma_1,\ldots,\gamma_r)=\sum_{k=1}^r \R^{\geq 0}(\gamma_1+\cdots+\gamma_k).
$$
In particular, the weight $\gamma_1+\cdots+\gamma_k$ is dominant for any $k=1,\ldots, r$.
\end{prop}

\medskip

We know then that $\Phi_K(X)=\frac{1}{\|\gamma_1\|^2}\sum_{k=1}^r (t_k)^2 \gamma_k$ 
belongs to the Weyl chamber $\tgot^*_+$ if $X=\sum_{k=1}^r t_k X_k$ belongs to the 
chamber $\agot_+$. Hence, the moment map $\Phi_K:\pgot\to \kgot^*$ defines a one to one map between $\agot_+$ 
and the the cone $\sum_{k=1}^r \R^{\geq 0}(\gamma_1+\cdots+ \gamma_k)\subset\tgot^*_+$. Using now the 
fact that $\agot_+$ and $\tgot^*_+$ are respectively fundamental domains for the $K$-action on 
$\pgot$ and $\kgot^*$, we get the following

\begin{prop}
$\bullet$ The Kirwan polyhedral cone $\Delta_K(\pgot)$ is equal to 
\begin{equation}\label{eq:delta-K-p}
\sum_{k=1}^r \R^{\geq 0}(\gamma_1+\cdots+ \gamma_k).
\end{equation}

$\bullet$ The $K$-Hamiltonian space $(\pgot,\Omega_\pgot)$ is without multiplicities : for any 
$\xi\in\kgot^*$, the fiber $\Phi^{-1}_K(K\cdot\xi)\subset \pgot$ is a $K$-orbit.

\end{prop}

We can summarize the results of Sections \ref{sec:proper} and \ref{sec:delta-K-p} in the following

\begin{theo}\label{theo:equivalence}

Let $V_{\Lambda}^K\otimes S(\pgot^+)$ be the admissible $K$-representation attached to 
$\Lambda\in \Ccal_{\rm hol}$. Let $H$ be a closed connected Lie subgroup of $K$. Let 
$\Phi_H:\Ocal_\Lambda\to\hgot^*$ be the moment map relative to the action of $H$ on the coadjoint 
orbit $\Ocal_\Lambda$. The following statement are equivalent:
\begin{enumerate}
\item The map $\Phi_H:\Ocal_\Lambda\to\hgot^*$ is {\em proper}.
\item The $H$-multiplicities in $V_{\Lambda}^K\otimes S(\pgot^+)$ are {\em finite}.
\item The subalgebra $S(\pgot^+)^H$ formed by the $H$-invariant elements is reduced to the constants.
\item We have 
\begin{equation}\label{condition-admissible}
\left(\sum_{k=1}^r \R^{\geq 0}(\gamma_1+\cdots+ \gamma_k)\right)\bigcap K\cdot\hgot^\perp = \{0\}.
\end{equation}
\end{enumerate}
\end{theo}

\begin{rem}
Note that the condition (\ref{condition-admissible}) holds trivially when $H=K$ since then 
$K\cdot\hgot^\perp=\{0\}$. When $H$ is equal to the center $Z(K)\subset K$, the set 
$K\cdot\hgot^\perp= {\rm Lie}(Z(K))^\perp$ intersects 
$\sum_{k=1}^r \R^{\geq 0}(\gamma_1+\cdots+ \gamma_k)$ only at $0$ since 
$\langle\gamma_k,z_o\rangle=1$ for all $k=1,\ldots,r$.
\end{rem}

\bigskip

We finish this section by considering the example of $\SU(p,q)$, with $1\leq p\leq q$. 
A maximal compact subgroup of $\SU(p,q)$ is $K=\S(\U(p)\times\U(q))$. 
The maximal torus $T\subset K$ is composed by the diagonal matrices. The dual of 
its Lie algebra is 
$$
\tgot^*:=\{(x_1,\ldots, x_{p+q})\in \R^{p+q}\ \vert \ \sum_j x_j=0\}.
$$
The vector space $\pgot^+$ is the complex vector space ${\rm M}_{p,q}(\C)$ of 
complex $p\times q$ matrices. The action of $K=\S(\U(p)\times\U(q))$ on $\pgot^+={\rm M}_{p,q}(\C)$
is defined by $(g,h)\cdot M=gMh^{-1}$.

The Weyl chamber relative to a system of positive compact roots $\Rgot_c^+$  is 
$$
\tgot^*_+:=\left\{(x_1,\ldots, x_{p+q})\in \tgot^*\ \vert \ x_1\geq \cdots \geq x_p\ {\rm and}\ 
x_{p+1}\geq \cdots \geq x_{p+q}\right\}.
$$
The Weyl chamber relative to a system of positive roots $\Rgot^+_{\rm hol}$ is 
$$
\overline{\Chol}:=\left\{(x_1,\ldots, x_{p+q})\in \tgot^*\ \vert \ x_1\geq \cdots \geq x_p\geq 
x_{p+1}\geq \cdots \geq x_{p+q}\right\}.
$$
A family of strongly orthogonal roots is $\Psi=\{\gamma_1,\ldots,\gamma_p\}$ where\footnote{Here 
$\{e_1,\ldots,e_{p+q}\}$ is the canonical basis of $\R^{p+q}$.} 
$$
\gamma_j=e_j-e_{p+q-j+1}.
$$ 
Hence the cone $\sum_{k=1}^p \R^{\geq 0}(\gamma_1+\cdots+ \gamma_k)$ is equal to 
$$
\Dcal:=\Big\{(x_1,\ldots, x_p,\underbrace{0,\ldots,0}_{q-p\ {\rm times}},-x_p,\ldots,-x_1)\ \vert\ 
x_1\geq \cdots \geq x_p\geq 0\Big\}.
$$

Let us consider the normal subgroups $\SU(p)$ and $\SU(q)$ of $K$. If $H=\SU(p)$, it is not hard to see
 that 
$$
\hgot^\perp\cap\tgot^*=\Big\{(\underbrace{x,\ldots,x}_{p\ {\rm times}},y_1,\ldots, y_{q})\ 
\vert \ px +\sum_j y_j=0\Big\}.
$$
so $\hgot^\perp\cap \Dcal$ contains the non-zero element 
$(\underbrace{1,\ldots,1}_{p\ {\rm times}},\underbrace{0,\ldots,0}_{q-p\ {\rm times}},-1,\ldots,-1)$.
Thus we know from Theorem \ref{theo:equivalence} that 
\begin{enumerate}
\item the holomorphic discrete series representations 
of $\SU(p,q)$ does not have a admissible restriction to $\SU(p)$,
\item the algebra $S({\rm M}_{p\times q}(\C))$ has an homogeneous $\SU(p)$-invariant element 
with strictly positive degree.
\end{enumerate}

Consider now the case where $H=\SU(q)$ with $p<q$. We see
 that 
$$
\hgot^\perp\cap\tgot^*=\Big\{(x_1,\ldots,x_p,\underbrace{y,\ldots,y}_{q\ {\rm times}})\ 
\vert \ \sum_k x_k +qy=0\Big\}.
$$
and that $\hgot^\perp\cap \Dcal=\{0\}$. From Theorem \ref{theo:equivalence} we have then that, if $p<q$,
\begin{enumerate}
\item the holomorphic discrete series representations 
of $\SU(p,q)$ have an admissible restriction to $\SU(q)$,
\item the algebra $S({\rm M}_{p\times q}(\C))$ does not have an homogeneous $\SU(q)$-invariant element 
with strictly positive degree.
\end{enumerate}

\section{Multiplicities of the discrete series}\label{sec:disc-series}

Let $G$ be a real, connected, semi-simple Lie group with finite center.  
Let $K$ be a maximal compact subgroup of $G$, and $T$ be a 
maximal torus in $K$. For the remainder of this section, we 
assume that $T$ is a Cartan subgroup of $G$. The discrete series of 
$G$ is then non-empty and is parametrized by a subset $\widehat{G}_{d}$ in 
the dual $\tgot^*$ of the Lie algebra of $T$ \cite{Harish-Chandra65-66}.

Let us fix some notation.
Let $\Rgot_c\subset\Rgot\subset\wedge^{*}$ be respectively the set of 
(real) roots for the action of $T$ on $\kgot\otimes\C$ and 
$\ggot\otimes\C$. We choose a system of positive roots $\Rgot_{c}^{+}$ for 
$\Rgot_{c}$, we denote by $\tgot^{*}_{+}$ the corresponding Weyl chamber, 
and we let $\rho_{c}$ be half the sum of the elements of $\Rgot_{c}^{+}$. 

An element $\lambda\in \tgot^{*}$ is called {\em regular} if 
$(\lambda,\alpha)\neq 0$ for every $\alpha\in\Rgot$, or equivalently,  
if the stabilizer subgroup of $\lambda$ in $G$ is $T$. 
Given a system of positive roots $\Rgot^{+}$ for $\Rgot$, consider 
the subset $\wedge^{*}+\frac{1}{2}\sum_{\alpha\in\Rgot^{+}}\alpha$ of
$\tgot^*$. It does not depend on the choice of $\Rgot^{+}$, and we denote it 
by $\wedge^{*}_{\rho}$ \cite{Duflo77}. 
 
\medskip 
 
The discrete series of $G$ are parametrized by 
\begin{equation}\label{G-hat}
    \widehat{G}_{d}:=\{\lambda\in\tgot^{*},\lambda\ {\rm regular}\ 
    \}\cap \wedge^{*}_{\rho}\cap \tgot^{*}_{+}\ .
\end{equation}

An element $\lambda\in \widehat{G}_{d}$  determines a choice 
$\Rgot^{+,\lambda}$ of positive roots for the $T$-action on 
$\ggot\otimes\C$ : $\alpha\in \Rgot^{+,\lambda}\Longleftrightarrow 
(\alpha,\lambda)>0$. We have $\Rgot^{+,\lambda}=\Rgot^+_c\cup \Rgot^{+,\lambda}_n$
and we define 
$$
\rho_n(\lambda):=\frac{1}{2}\sum_{\beta\in \Rgot^{+,\lambda}_n}\beta,
$$
Note that the {\em Blattner parameter} 
$$
\Lambda(\lambda):= \lambda-\rho_c +\rho_n(\lambda)
$$ 
is a dominant weight for any $\lambda\in \widehat{G}_{d}$. We work in this section under 
Condition (\ref{condition-lambda}), which states that $\beta\in \Rgot^{+,\lambda}_n\Longleftrightarrow 
(\beta,\Lambda(\lambda))>0$. This implies in particular, that the dominant weight $\Lambda(\lambda)$ does not 
belong to the non-compact walls. 

Let us consider the coadjoint orbit $\Ocal_{\Lambda(\lambda)}:=G\cdot \Lambda(\lambda)$. 
It is a $G$-Hamiltonian manifold  which is prequantized by the line bundle 
$\Lcal_{\Lambda(\lambda)}:=G\times_{K_{\Lambda(\lambda)}}\C_{\Lambda(\lambda)}$. We equip
$\Ocal_{\Lambda(\lambda)}$ with the $G$-invariant {\em almost complex structure} $J_{\Lambda(\lambda)}$ 
which is characterized by the following fact. The bundle 
$\T^{1,0}\Ocal_{\Lambda(\lambda)}\to \Ocal_{\Lambda(\lambda)}$ is equal, above 
$\Lambda(\lambda)\in\Ocal_{\Lambda(\lambda)}$, to the $T$-module
$$
\sum_{\substack{\alpha\in \Rgot_c\\ \langle\alpha,\Lambda(\lambda)\rangle > 0}}\ggot_\alpha \ \oplus
\underbrace{\sum_{\substack{\beta\in \Rgot_n\\ \langle\beta,\Lambda(\lambda)\rangle < 0}}
\ggot_\beta}_{\pgot(\lambda)^-} .
$$
Similarly we note $\pgot(\lambda)^+:=\sum_{\beta\in \Rgot^{+,\lambda}_n}\ggot_\beta\subset\pgot\otimes\C$. 
Note that the {\em almost complex structure} $J_{\Lambda(\lambda)}$ is compatible with the symplectic structure on 
$\Ocal_{\Lambda(\lambda)}$, but in genaral $J_{\Lambda(\lambda)}$ is not integrable. 

\medskip

Let $\Hcal_\lambda$ be a discrete series representation attached to $\lambda\in \widehat{G}_{d}$. Recall that 
the restriction $\Hcal_\lambda\vert_K$ is an admissible representation.

The main result of this section is  
\begin{theo}\label{theo:qfor-lambda-lambda} 
If $\lambda\in \widehat{G}_{d}$ satisfy condition (\ref{condition-lambda}) we have 
\begin{equation}\label{eq:theo-deux}
\Hcal_\lambda\vert_K=\qfor_K(\Ocal_{\Lambda(\lambda)}).
\end{equation}
\end{theo}

Like we did before, if we use (\ref{eq:theo-deux}) together with the property {\bf [P2]}, we get Theorem \ref{theo:intro-deux}.

The proof of Theorem \ref{theo:qfor-lambda-lambda} is similar to the proof of Theorem \ref{theo:formal-lambda}. 
We introduce, like in Section \ref{sec:sigma-L}, a $K$-transversaly elliptic 
symbol $\sigma_{\Lambda(\lambda)}$ on $\Ocal_{\Lambda(\lambda)}$ built from the 
data $(\Lcal_{\Lambda(\lambda)}, J_{\Lambda(\lambda)})$ and the moment map 
$\Phi_K: \Ocal_{\Lambda(\lambda)}\to\kgot^*$. The same deformation argument as the one 
used in Section \ref{sec:deformation} shows that 
$$
\indice^K\left(\sigma_{\Lambda(\lambda)}\right)=\qfor_K(\Ocal_{\Lambda(\lambda)}).
$$
Thus  Theorem \ref{theo:qfor-lambda-lambda} follows from the following

\begin{prop}\label{prop-indice-sigma-general} 
If $\lambda\in \widehat{G}_{d}$ satisfy condition (\ref{condition-lambda}), we have 
$$
\indice^K(\sigma_{\Lambda(\lambda)})=\Hcal_\lambda\vert_K\quad {\rm in}\quad \Rforc(K).
$$
\end{prop}

\subsection{Proof of Proposition \ref{prop-indice-sigma-general} }
The proof is an adaptation to the proof of Proposition \ref{prop-indice-sigma-Lambda}. Here we consider the 
the $K$-invariant diffeomorphism
\begin{equation}
\what{\Upsilon}: \Ocal_{\Lambda(\lambda)}\longrightarrow \what{\Ocal}_{\Lambda(\lambda)}:=K\times_{K_{\Lambda(\lambda)}} \pgot, 
\end{equation}
defined by $\what{\Upsilon}(ke^X \cdot\Lambda(\lambda)):=[k,X]$.

The data $(J_{\Lambda(\lambda)},\Lcal_{\Lambda(\lambda)},\Hcal,\sigma_{\Lambda(\lambda)})$, transported to the 
manifold $\what{\Ocal}_{\Lambda(\lambda)}$ through $\what{\Upsilon}$, is denoted 
$(\what{J}_{\Lambda(\lambda)},\what{\Lcal}_{\Lambda(\lambda)},\what{\Hcal},\what{\sigma}_{\Lambda(\lambda)})$. 
The line bundle $\what{\Lcal}_{\Lambda(\lambda)}$ is the pull-back of the line bundle 
$K\times_{K_{\Lambda(\lambda)}}\C_{\Lambda(\lambda)}\to K\cdot{\Lambda(\lambda)}$ 
to $\what{\Ocal}_{\Lambda(\lambda)}$.

The tangent bundle $\T\what{\Ocal}_{\Lambda(\lambda)}$ is $K$-equivariantly isomorphic to 
$K\times_{K_{\Lambda(\lambda)}} (\rgot_\lambda\oplus\T\pgot)$ where $\rgot_\lambda:=[\kgot,\Lambda(\lambda)]$ is the 
$K_{\Lambda(\lambda)}$-invariant complement of $\kgot_{\Lambda(\lambda)}$.

Let $J_\lambda$ be the (linear) complex structure on the vector space $\pgot$ such that 
$(\pgot,J_\lambda)\simeq \pgot(\lambda)^+$. Note that $J_\lambda$ is $K_{\Lambda(\lambda)}$-invariant since 
$\lambda$ satisfies condition (\ref{condition-lambda}).

Let $J_{K\cdot\Lambda(\lambda)}\vert_e$ be the (linear) $K_{\Lambda(\lambda)}$-invariant complex structure 
on the vector space $\rgot_\lambda$ defined by the K\"ahler structure $J_{K\cdot\Lambda(\lambda)}$ on the coadjoint orbit 
$K\cdot\Lambda(\lambda)$.

We consider on $\what{\Ocal}_{\Lambda(\lambda)}$ the 
following $K$-equivariant data:
\begin{enumerate}
\item The almost complex structure $\what{J}'_\Lambda$ such that 
$$
\what{J}'_\Lambda\vert_{(e,v)}=J_{K\cdot\Lambda(\lambda)}\vert_e\times -J_\lambda\quad 
{\rm for\ every}\quad v\in\pgot.
$$

\item The vector field $\what{\Hcal}'$ defined by: $\what{\Hcal}'_{[k, v]}=-\left(0,k\cdot [\Lambda(\lambda),v]\right)$ 
for $[k,v]\in\what{\Ocal}_{\Lambda(\lambda)}$.
\end{enumerate}

\begin{defi}\label{defi-push-general}
We consider on $\what{\Ocal}_{\Lambda(\lambda)}$ the symbols:
\begin{itemize} 
\item $\what{\tau}'_{\Lambda(\lambda)}:=\Thom(\what{\Ocal}_{\Lambda(\lambda)},\what{J}'_\Lambda)\otimes 
\what{\Lcal}_{\Lambda(\lambda)}$,

\item $\what{\sigma}'_{\Lambda(\lambda)}$ which is the symbol $\what{\tau}'_{\Lambda(\lambda)}$ 
pushed by the vector field $\what{\Hcal}'$ (see Def. \ref{def:sigma-Lambda}). 
\end{itemize}
\end{defi}

\begin{prop}\label{prop:sigma-sigma-gene}
$\bullet$ The symbol $\what{\sigma}'_{\Lambda(\lambda)}$ is a $K$-transversaly elliptic symbol on 
$\what{\Ocal}_{\Lambda(\lambda)}$.

$\bullet$ If $\Ucal$ is a sufficiently small $K$-invariant neighborhood  of $K\times_{K_{\Lambda(\lambda)}} \{0\}$ in 
$\what{\Ocal}_{\Lambda(\lambda)}$, the restrictions $\what{\sigma}_{\Lambda(\lambda)}\vert_\Ucal$ and 
$\what{\sigma}'_{\Lambda(\lambda)}\vert_\Ucal$ define the same class in $\K_K(\T_K\Ucal)$.
\end{prop}

\begin{proof} 
The proof works as the proof of Proposition \ref{prop:sigma-sigma}.
\end{proof}

\medskip

Proposition \ref{prop:sigma-sigma-gene} shows that
$\indice^K(\sigma_{\Lambda(\lambda)})=\indice^K(\what{\sigma}_{\Lambda(\lambda)})=\indice^K(\what{\sigma}'_{\Lambda(\lambda)})$. In order to compute 
$\indice^K(\what{\sigma}_{\Lambda(\lambda)}')$, we use the induction
morphism 
$$
i_*:\K_{K_\Lambda}(\T_{K_\Lambda}\pgot)\longrightarrow\K_{K}(\T_{K}(\what{\Ocal}_{\Lambda(\lambda)}))
$$ 
defined by Atiyah in \cite{Atiyah74} (see also \cite{pep-RR}[Section 3]). Here
$i_*$ differs from the induction morphism $j_*$ used in Section \ref{sec:Q-direct}, by the 
isomorphism 
$$
\K_{K}(\T_{K}(\what{\Ocal}_{\Lambda(\lambda)}))\simeq \K_{K}(\T_{K}(K\!\cdot\!\Lambda(\lambda)\times\pgot))
$$
induced by the $K$-diffeomorphism $\what{\Ocal}_{\Lambda(\lambda)}\simeq K\!\cdot\!\Lambda(\lambda)\times\pgot$, 
$[k,X]\mapsto (k\cdot\Lambda(\lambda), k\cdot X)$.

Let $\Thom(\pgot(\lambda)^-)$ be the $K_\Lambda(\lambda)$-equivariant Thom symbol of the complex vector space $\pgot(\lambda)^-\simeq (\pgot,-J_\lambda)$. 
Let $\widetilde{\Lambda(\lambda)}$ be the vector field on $\pgot$ which is generated by $\Lambda(\lambda)\in\kgot^*\simeq\kgot$. 
Let 
$$
\Thom^{\Lambda(\lambda)}(\pgot(\lambda)^-)
$$ 
be the symbol $\Thom(\pgot(\lambda)^-)$ pushed by the 
vector field $\widetilde{\Lambda(\lambda)}$ (see Definition \ref{defi-push}). 

Since $\Lambda(\lambda)$ does not belongs to the non-compact walls (see condition (\ref{condition-lambda})), the vector field 
$\widetilde{\Lambda(\lambda)}$ vanishes only at $0\in \pgot$: hence the symbol 
$\Thom^{\Lambda(\lambda)}(\pgot(\lambda)^-)$ is $K_{\Lambda(\lambda)}$-transversaly elliptic. 

One checks easily that
\begin{equation}\label{eq:j-sigma-Lambda-gene}
(i_*)^{-1}(\what{\sigma}'_{\Lambda(\lambda)})=\Thom^{\Lambda(\lambda)}(\pgot(\lambda)^-)\, 
\otimes\wedge^{\bullet}_\C\kgot/\kgot_{\Lambda(\lambda)}\otimes\C_{\Lambda(\lambda)}.
\end{equation}

Let  $\indl : \fgene(K_{\Lambda(\lambda)})^{K_{\Lambda(\lambda)}}\longrightarrow \fgene(K)^{K}$ be the 
induction map introduced in (\ref{eq.induction}). Equality (\ref{eq:j-sigma-Lambda-gene}) 
and the commutative diagram (\ref{diagram}) give 
\begin{eqnarray*}
\indice^K(\what{\sigma}'_{\Lambda(\lambda)})&=&\indl\left(
\indice^{K_{\Lambda(\lambda)}}\Big(\Thom^{\Lambda(\lambda)}(\pgot(\lambda)^-)\Big) 
\otimes\wedge^{\bullet}_\C\kgot/\kgot_{\Lambda(\lambda)}\otimes\C_{\Lambda(\lambda)}\right)\\
&=&\indT\left(
\indice^{T}\Big(\Thom^{\Lambda(\lambda)}(\pgot(\lambda)^-)\Big) 
\otimes\wedge^{\bullet}_\C\kgot/\tgot\otimes\C_{\Lambda(\lambda)}\right)
\end{eqnarray*}

In the last equality, we use two facts (see \cite{pep-RR})  :

$\bullet$ Since the symbol $\Thom^{\Lambda(\lambda)}(\pgot(\lambda)^-)$ is $T$-transversally elliptic, 
the index \break   
$\indice^{K_{\Lambda(\lambda)}}\left(\Thom^{\Lambda(\lambda)}(\pgot(\lambda)^-)\right)$ is $T$-admissible, and 
its restriction to $T$ is equal to  $\indice^{T}\left(\Thom^{\Lambda(\lambda)}(\pgot(\lambda)^-)\right)$.

$\bullet$ For any 
$K_{\Lambda(\lambda)}$-module $E$ which is $T$-admissible we have 
$$
\indl\left( E\otimes\wedge^{\bullet}_\C\kgot/\kgot_{\Lambda(\lambda)}\right)=
\indT\left( E\vert_T\otimes\wedge^{\bullet}_\C\kgot/\tgot\right).
$$

We know from \cite{pep-RR}[Section 5.1] that the $T$-index of 
$\Thom^{\Lambda(\lambda)}(\pgot(\lambda)^-)$ is equal 
to the symmetric algebra $S(\pgot(\lambda)^+)$ viewed as a $T$-module. Here we use in a crucial 
way Condition (\ref{condition-lambda}): for every weight $\beta$ relative to the $T$-action on the 
complex vector spaces $\pgot(\lambda)^-$, we have $(\beta,\Lambda(\lambda))<0$. 
The $T$-module $S(\pgot(\lambda)^+)$ is denoted 
$$
\Big[\prod_{\beta\in\Rgot^{+,\lambda}_n}(1-t^\beta)\Big]^{-1}_\lambda\ \in\ \Rforc(T).
$$
in \cite{pep-ENS}. So we have proved that   
\begin{equation}\label{eq:1}
\indice^K(\sigma_{\Lambda(\lambda)})= \indT\left(\Big[\prod_{\beta\in\Rgot^{+,\lambda}_n}(1-t^\beta)\Big]^{-1}_\lambda
\otimes\C_{\Lambda(\lambda)}\otimes\wedge^{\bullet}_\C\kgot/\tgot\right).
\end{equation}

We have proved in \cite{pep-ENS}[Section 5.1] that the Blattner formulas \cite{Hecht-Schmid} 
which computes the $K$-multiplicities of the discrete series representation $\Hcal_\lambda$ are equivalent 
to the following relation
\begin{equation}\label{eq:2}
\Hcal_\lambda\vert_K=\HolT\left(\Big[\prod_{\beta\in\Rgot^{+,\lambda}_n}(1-t^\beta)\Big]^{-1}_\lambda
\otimes\C_{\Lambda(\lambda)}\right)\quad {\rm in}\quad \Rforc(K), 
\end{equation}
where the ``holomorphic'' induction map $\HolT$ is equal to $\indT(- \otimes\wedge^{\bullet}_\C\kgot/\tgot)$.

We see that (\ref{eq:1}) and (\ref{eq:2}) complete the proof of Proposition \ref{prop-indice-sigma-general}.

\subsection{Examples}\label{section:condition-lambda}

\subsubsection{The case of $\ {\rm Sp}(2,\R)$}

We examined this case in  Example (\ref{ex:sp(2)}). Let $\theta_1,\theta_2$ be the $\Z$-basis of of the 
lattice $\wedge^*$. The set of compact roots is $\Rgot_c=\{\pm(\theta_1-\theta_2)\}$, and the set of 
non-compact roots is $\Rgot_n=\{\pm(\theta_1+\theta_2),\pm 2\theta_1,\pm 2\theta_2\}$. We 
choose $\theta_1-\theta_2$ as the positive compact root, hence $\tgot^*_+=\{\theta_1\geq\theta_2\}$.

The set of strongly elliptic elements in the Weyl chamber $\tgot^*_+$ has four chambers (see Figure 
(\ref{figure-SP4-2})): 
$\Ccal_1=\{\theta_1\geq\theta_2> 0 \}$, 
$\Ccal_2=\{\theta_1>-\theta_2> 0 \}$, $\Ccal_3=\{-\theta_2>\theta_1 >0 \}$, and 
$\Ccal_4=\{-\theta_2\geq-\theta_1> 0 \}$.

\begin{figure}[h]\label{figure-SP4-2}
\centerline{
\scalebox{0.3}{\input{SP4.2.pstex_t}}
}
\caption{Chambers for  ${\rm Sp}(2,\R)$}
\end{figure}

For $\lambda\in\tgot^*_+$ which is regular, the term  $\rho_n(\lambda)$ only depends of the chamber $\Ccal_i$ where 
$\lambda$ stands: let us denoted it $\rho_n(\Ccal_i)$. 

We check that $-\rho_c+\rho_n(\Ccal_i)\in \overline{\Ccal_i}$ for $i=2,3$. Hence, for $i=2,3$ 
and any Harish-Chandra parameter $\lambda\in \Ccal_i$, we have 
$\Lambda(\lambda)=\lambda-\rho_c+\rho_n(\Ccal_i)\in\Ccal_i$.

We know already that any regular weight of the holomorphic chamber $\Ccal_1$ satisfies 
condition (\ref{condition-lambda}). It is also the case for the {\em anti-holomorphic} 
chamber $\Ccal_4$.

Finally we see that condition (\ref{condition-lambda}) holds for any Harish-Chandra parameter of 
a discrete series of ${\rm Sp}(2,\R)$.

\subsubsection{The case of $\ {\rm Sp}(4,\R)$}

Let $\theta_1,\cdots,\theta_4$ be the canonical basis of $\R^4\simeq\tgot^*$. The 
compact positive roots are $\theta_i-\theta_j, 1\leq i<j\leq 4$, so that the corresponding 
Weyl chamber is $\tgot^*_+:=\{\lambda_1\geq\lambda_2\geq\lambda_3\geq\lambda_4\}$, and 
$\rho_c=\frac{1}{2}(3,1,-1,-3)$. The set of  non-compact roots is 
$\{2\theta_i\}\cup\{\theta_i+\theta_j, i<j\}$.

We consider the chamber $\Ccal:=\{\lambda_1 \geq \lambda_2 > -\lambda_4 > \lambda_3>0\}$ of the 
Weyl chamber $\tgot^*_+$. We have $\rho_n(\Ccal)=\frac{1}{2}(5,5,3,-1)$ and then
$$
-\rho_c+\rho_n(\Ccal)=(1,2,2,1).
$$
We check that $\lambda=(5,3,1,-2)$ is a Harish-Chandra parameter belonging to $\Ccal$, but 
$\Lambda(\lambda)=(6,5,3,-1)$ does not belong to $\overline{\Ccal}$.

\subsubsection{The case of $\ {\rm SU}(3,2)$}

Let $T$ be the torus of $\ {\rm SU}(3,2)$ formed by all the diagonal matrices. The dual of Lie algebra of $T$ 
is $\tgot^*=\{(\lambda_1,\cdots,\lambda_5)\in\R^5\ \vert\ \sum_i\lambda_i=0\}$. Let $e_1,\cdots, e_5$ be the 
canonical basis of $\R^5$. The choice of positive compact roots $\Rgot^+_c$ is 
$\{e_1-e_2,e_1-e_3,e_2-e_3,e_4-e_5\}$ so that the Weyl chamber is 
$$
\tgot^*_+:=\Big\{ \lambda_1\geq\lambda_2\geq \lambda_3\ {\rm and}\ \lambda_4\geq\lambda_5\Big\}.
$$
We have $\rho_c=(1,0,-1,\frac{1}{2},-\frac{1}{2})$. The non-compact roots are $\pm(e_i - e_j),\ i=1,2,3,\ j=4,5$. 

 $\bullet$ Let $\lambda=(3,1,-1,0,-3)$ be in the chamber 
$\Ccal_1:=\{\lambda_1\geq\lambda_2> \lambda_4 > \lambda_3>\lambda_5\}$.
We have $\rho_n(\Ccal_1)=(1,1,0,-\frac{1}{2},-\frac{3}{2})$, and then 
$$
\Lambda(\lambda)=\lambda-\rho_c+\rho_n(\Ccal_1)=(3,2,0,-1,-4)
$$
is a regular element which does not belong to $\Ccal_1$.

$\bullet$ Let us consider the chamber $\Ccal_2:=\{\lambda_1>\lambda_4> \lambda_2 >\lambda_5>\lambda_3\}$. 
We see that $\rho_n(\Ccal_2)=\rho_c$, hence any Harish-Chandra parameter of the chamber $\Ccal_2$ satisfies 
condition (\ref{condition-lambda}).

\section{Appendices}\label{sec:appendix}

Let $G$ be a connected real semi-simple Lie group with finite center. Let $K$ be a maximal compact 
Lie subgroup of $G$. Let $T$ be a maximal torus in $K$. Let $\tgot,\kgot,\ggot$ be the 
respective Lie algebras of $T,K,G$. 
We assume that $\tgot$ is a Cartan subalgebra of $\ggot$. 

In Appendix A, we use the identification $X\mapsto B_\ggot(X,-), \ggot\stackrel{\sim}{\longrightarrow}\ggot^*$ 
given by the Killing form. Hence the {\em coadjoint} orbits of $G$ considered in the previous sections will 
be replaced by {\em adjoint} orbits. 

\subsection{Appendix A}\label{sec:appendix-A}

Let $\Ocal= G\cdot \lambda$ be an adjoint orbit of $G$ passing through 
$\lambda\in\tgot$. Let $K\cdot\mu$ be an adjoint orbit of $K$ passing 
through $\mu\in\tgot$. We consider the maps
$$
\Phi_t: \Ocal\times K\cdot\mu\longrightarrow \kgot, \quad t\in [0,1]
$$
defined by $\Phi_t(m,\xi)=\pi_\kgot(m)-t\xi$. Here $\pi_\kgot:\ggot\to\kgot$ is 
the orthogonal projection. The maps $\Phi_t,t\in [0,1]$ generates the vector 
fields $\Hcal_t,t\in [0,1]$ on $\Ocal\times K\cdot\mu$ by 
$\Hcal_t(n)=(V\Phi_t(n))\vert_n$ for $n\in \Ocal\times K\cdot\mu$.

The aim of this section is the following 

\begin{prop}
There exists a compact subset $\Kcal$ of $\Ocal$ such that 
$$
\left\{ \Hcal_t=0\right\}\subset \Kcal\times K\cdot\mu
$$
for any $t\in [0,1]$.
\end{prop}

\begin{proof} The proof is given in \cite{pep-ENS}[Section 5.3] in the case where 
$\lambda$ is a regular element of $\ggot$.
Here we propose another proof, which is technically simpler, that 
was communicated to us by Mich\`ele Vergne. 

By definition, we have 
$$
\Hcal_t(m,\xi)=-\Big([\pi_\kgot(m)-t\xi, m],  [\pi_\kgot(m),\xi ]\Big)\quad\in \ 
\T_n\Ocal\times \T_\xi(K\cdot\mu)
$$
Let us denote $\Ccal_t$ the subset $\left\{ \Hcal_t=0\right\}$. We have 
\begin{eqnarray*}
\Ccal_t &=&\left\{(m,\xi)\in \Ocal\times K\cdot\mu\ \mid \ [\pi_\kgot(m)-t\xi, m]=0\ 
\mathrm{and}\   [\pi_\kgot(m),\xi ]=0 \right\} \\
&=& K\cdot\left\{(m,\mu)\in \Ocal\times K\cdot\mu\ \mid \ [\pi_\kgot(m)-t\mu, m]=0\ 
\mathrm{and}\   [\pi_\kgot(m),\mu]=0\right\}.
\end{eqnarray*}
The condition $[\pi_\kgot(m),\mu]=0$ means that $\pi_\kgot(m)$ belongs to the 
subalgebra $\kgot_\mu$ that stabilizes $\mu\in\tgot$. 
We have $\kgot_\mu=K_\mu\cdot\tgot$, hence $\Ccal_t \subset  K\cdot \Dcal_t \times K\cdot\mu$
where 
$$
\Dcal_t=
\left\{m \in\Ocal \ \mid \ \pi_\kgot(m)\in \tgot \ \mathrm{and}\  
m\in\ggot_{\pi_\kgot(m)-t\mu}  \right\}.
$$
Here $\ggot_{\pi_\kgot(m)-t\mu}$ is the subalgebra that stabilizes $\pi_\kgot(m)-t\mu$. 
The proof will be settled if one proves that 
$\cup_{t\in[0,1]}\Dcal_t$ is contained in a compact subset of $\Ocal$.

The subalgebras $\ggot_X, X\in\tgot$ describe a finite subset that we enumerate 
$\ggot_i, i=1,\ldots, r$. For each subalgebra $\ggot_i$, 
let $G_i$ be the corresponding closed connected subgroup of $G$. Note that $\tgot$ 
is contained in each $\ggot_i$, and that the center 
$z(\ggot_i)$ of $\ggot_i$ is contained in (the Cartan subalgebra) $\tgot$. Note 
that the condition $\ggot_{\pi_\kgot(m)-t\mu}=\ggot_i$ implies that 
$\pi_\kgot(m)-t\mu\in z(\ggot_i)$. It gives that $\Dcal_t\subset 
\cup_{i=1}^r \Dcal_t^i$ with
$$
\Dcal_t^i=
\left\{m \in\Ocal\cap \ggot_i \ \mid \ \pi_\kgot(m)-t\mu\in z(\ggot_i)  \right\}.
$$
It is a classical result that the intersection $\Ocal\cap \ggot_i$ is equal to a 
{\em finite} collection of adjoint $G_i$ orbit:
$$
\Ocal\cap \ggot_i=\bigcup_{\alpha\in A_i} G_i\cdot \alpha.
$$
Let $\pi_i:\ggot\to z(\ggot_i)$ be the orthogonal projection. If 
$\pi_\kgot(m)-t\mu\in z(\ggot_i)$, we have 
\begin{eqnarray*}
\pi_\kgot(m)-t\mu &=&\pi_i\Big(\pi_\kgot(m)-t\mu\Big)\\
&=& \pi_i(m)-t\pi_i(\mu).
\end{eqnarray*}
But the map $\pi_i$ is constant on each connected component $G_i\cdot \alpha$. So finally,
$$\Dcal_t^i=
\bigcup_{\alpha\in A_i} \left\{m \in G_i\cdot \alpha \ \mid \ \pi_\kgot(m)-t\mu
=\pi_i(\alpha)-t\pi_i(\mu). \right\}.
$$
and then
\begin{eqnarray*}
\Dcal_t^i&=&
\bigcup_{ \alpha\in A_i} G_i\cdot \alpha\cap  \pi_\kgot^{-1}\left(\theta_{i,\alpha,t}\right) \\
&\subset&\bigcup_{ \alpha\in A_i} \Ocal \cap \pi_\kgot^{-1}\left(\theta_{i,\alpha,t}\right)
\end{eqnarray*}
with $\theta_{i,\alpha,t}=\pi_i(\alpha)+t(\mu-\pi_i(\mu))$. We get finally that 
$$
\bigcup_{t\in[0,1]}\Dcal_t \subset  \Ocal \cap \pi_\kgot^{-1}(C)
$$
where
$$
C=\{\theta_{i,\alpha,t},\ t\in[0,1],\ i=1,\ldots, r, \ \alpha\in A_i\}
$$
is a compact subset of $\tgot$. Since the map $\pi_\kgot$ is {\em proper} when 
restricted to $\Ocal$, the set $\Ocal \cap \pi_\kgot^{-1}(C)$ is compact.

\end{proof}

\subsection{Appendix B}\label{sec:appendix-B}

Here we suppose that $G/K$ is an irreducible Hermitian symmetric spaces, and we use the notations
of Section \ref{sec:delta-K-p}. Our aim is the proof of Propositions \ref{prop:a-plus} and \ref{prop:appendixB-2}.
Our (classical) arguments uses the knowledge of the restricted root system $\Sigma$  and 
the Cayley transform.

We denote $(-,-)_\tgot$  the scalar product on $\tgot$ defined by :
 $(X,Y)_\tgot:=-B_\ggot(X,Y)$ for $X,Y\in\tgot$. Let $(-,-)_{\tgot^*}$ be the 
scalar product on $\tgot*$ which make the map $X\mapsto(X,-)_\tgot$, from $\tgot$ to $\tgot^*$,  
unitary.

Let $\agot=\sum_{j=1}^r \R X_j$ be the maximal abelian algebra of $\pgot$ attached to the 
maximal family $\Psi=\{\gamma_1,\ldots,\gamma_r\}$ of strongly orthogonal roots (see Section \ref{sec:delta-K-p}).

Let $\tgot_1\subset \tgot$ be the subspace orthogonal (for the duality) to the vector 
subspace spanned  by $\gamma_1,\ldots,\gamma_r$: $\tgot_1$ is also the centralizer of $\agot$ in $\tgot$. 
Let $\tgot_2\subset \tgot$ be the orthogonal of $\tgot_1$ (relatively to the 
scalar product on $\tgot$). We check easily that
$$
\tgot_2={\rm Vect}(H_1,\ldots,H_r).
$$
We have then the orthogonal decomposition $\tgot^*=\tgot_1^*\oplus\tgot_2^*$ with 
$\tgot_2^*={\rm Vect}(\gamma_1,\ldots,\gamma_r)$.

Let $\Rgot=\Rgot(\ggot_\C,\tgot_\C)$ be the roots system associated to the 
Cartan subalgebra $\tgot=\tgot_1\oplus\tgot_2$. Let $\Rgot^+_{\rm hol}=\Rgot^+_c\cup\Rgot^{+,z_o}_n$ 
be the system of positive roots consider in the introduction.  Let 
$\overline{\Chol}:=\left\{\xi\in\tgot^*\ \vert\ (\xi,\alpha)_{\tgot^*}\geq 0,\ \forall \alpha\in 
\Rgot^+_{\rm hol}\right\}$ be the corresponding Weyl chamber. 

Let $\pi':\tgot^*\to\tgot_2^*$ be the canonical 
projection, and let us consider  
$$
\Sigma':=\pi'(\Rgot)\setminus\{0\}\quad {\rm and}\quad  (\Sigma')^+:=\pi'(\Rgot^+_{\rm hol})\setminus\{0\}.
$$

We see that $\Ccal\cap {\rm Vect}(\gamma_1,\ldots,\gamma_r)=\Dcal$, with
\begin{equation}\label{eq:Dcal}
\Dcal:=\left\{\xi\in\tgot_2^*\ \vert\ (\xi,\alpha)_{\tgot^*}\geq 0,\ \forall \alpha\in (\Sigma')^+\right\}.
\end{equation}

Now we use the description of $\Sigma'$ given by Harish-Chandra and Moore.

\begin{prop}[\cite{Harish-Chandra55-56,Moore}]
$\bullet$ All the $\gamma_k$ have the same lenght.

$\bullet$ For any $i<j$, there is an $\alpha\in\Rgot^+_c$, such that $\pi'(\alpha)=\frac{1}{2}(\gamma_i-\gamma_j)$.

$\bullet$ They are two possibilities for $\Sigma':=\pi'(\Rgot)\setminus\{0\}$:
$$
\Sigma':=\Big\{\pm\hbox{$\frac{1}{2}$}(\gamma_i+\gamma_j),
\pm\hbox{$\frac{1}{2}$}(\gamma_i-\gamma_j),\ 1\leq i< j\leq r\Big\}
\cup\Big\{\pm\gamma_i, \ 1\leq i\leq r\Big\},
$$
or
$$
\Sigma':=\Big\{\pm\hbox{$\frac{1}{2}$}(\gamma_i+\gamma_j),
\pm\hbox{$\frac{1}{2}$}(\gamma_i-\gamma_j),\ 1\leq i< j\leq r\Big\}\cup
\Big\{\pm\hbox{$\frac{1}{2}$}\gamma_i, \pm\gamma_i, \ 1\leq i\leq r\Big\}.
$$
\end{prop}

\medskip

Since the $\gamma_k$ belongs to $(\Sigma')^+$, the last two point of Proposition shows that 
\begin{equation}\label{eq:sigma-reduit}
(\Sigma')^+=\left\{\hbox{$\frac{1}{2}$}(\gamma_i+\gamma_j),
\hbox{$\frac{1}{2}$}(\gamma_i-\gamma_j),\ 1\leq i< j\leq r\right\}
\cup\{\gamma_1, \ldots,\gamma_r \}\cup \Xi
\end{equation}
where $\Xi=\emptyset$ or $\Xi=\{\frac{1}{2}\gamma_1, \ldots,\frac{1}{2}\gamma_r \}$.

Since the $\gamma_k$ have the same lenght, it is now easy to see that the set $\Dcal$ defined in 
(\ref{eq:Dcal}) is equal to $\sum_{i=1}^r \R^{\geq 0}(\gamma_1+\cdots+\gamma_r)$. 
Thus the second point of Proposition \ref{prop:appendixB-2} is proved: we have 
\begin{equation}\label{eq:C-cap-gamma}
\Ccal\cap {\rm Vect}(\gamma_1,\ldots,\gamma_r)=\sum_{k=1}^r \R^{\geq 0}(\gamma_1+\cdots+\gamma_k).
\end{equation}

\begin{rem}
We know from (\ref{eq:H-gamma}) that $\|H_k\|=2\|\gamma_k\|^{-1}$. Thus, all the 
$H_k$ have the same lenght.
\end{rem}

Now, we go into the proof of Proposition \ref{prop:a-plus}: we will compute 
a fundamental domain $\agot_+$ for the action of $K$ on $\pgot$.

In the complex semi-simple algebra $\ugot:=\kgot\oplus i\pgot$, we consider the Cartan algebra 
$$
\hgot:=\tgot_1\oplus i\agot.
$$
that we equip with the scalar product $(X,Y)_\hgot:=-B_\ggot(X,Y), \quad \forall\ X,Y\in\hgot$.
We take on $\hgot^*$ the scalar product such that the map $\hgot\to\hgot^*, X\mapsto (X,-)_\hgot$ is orthogonal.

Let $\Rgot(\ggot_\C,\hgot_\C)\subset \hgot^*$ be the set of roots relative to the adjoint action 
of $\hgot_\C$ on $\ggot_\C$. The projection $\pi:\hgot^*\to(i\agot)^*$ sends 
$\Rgot(\ggot_\C,\hgot_\C)$ onto $\widetilde{\Sigma}\cup\{0\}$, where $\Sigma$ is the restricted 
root system, and $\xi\mapsto \tilde{\xi}, \agot^*\simeq (i\agot)^*$ is the 
one to one map defined by 
$\langle \tilde{\xi}, iX\rangle:=\langle \xi, X\rangle$.

The Cayley transform
$$
\cay:=\exp\left(-\frac{i\pi}{2}\ad\Big(\sum_{k=1}^r Y_k\Big)\right).
$$
is an automorphism of the complex Lie algebra $\ggot_\C$. One checks that $\cay(Y)=Y$ for any 
$Y\in\tgot_1$ and that 
$$
\cay(i X_k)=\frac{1}{2} H_k,\quad \forall k=1,\ldots,r.
$$
Hence the Cayley transform sends the subalgebra $\hgot$ onto 
the subalgebra $\tgot$. Moreover one checks easily that 
$\cay:\hgot\to\tgot$ is an orthogonal map, thus we know that all the $X_k$ have the same lenght. 
Let us denoted $\cay^*:\tgot^*\to\hgot^*$ the dual orthogonal map.

Since $\cay$ is an automorphism of $\ggot_\C$, the image of the root system 
$\Rgot:=\Rgot(\ggot_\C,\tgot_\C)$ by $\cay^*$ is equal to the root system $\Rgot(\ggot_\C,\hgot_\C)$. 
Since $\cay$ is the identity map on $\tgot_1$, we have $\cay^*(\Sigma')=\widetilde{\Sigma}$.

If we choose systems of positive roots such that 
$$
\Rgot(\ggot_\C,\hgot_\C)^+:=\cay^*(\Rgot^+_{\rm hol})\quad  {\rm and}\quad  \widetilde{\Sigma^+}:=\cay^*((\Sigma')^+),
$$
we get 
\begin{eqnarray*}
\agot_+&:=&\{X\in\agot\ \vert \langle \beta, X\rangle\geq 0, \ \forall\, \beta\in\Sigma^+\}\\
&=&\{X\in\agot\ \vert \langle \cay^*(\alpha), iX\rangle\geq 0, \ \forall\, \alpha\in (\Sigma')^+\}\\
&=&\{X=\sum_k a_k X_k\ \vert \sum_k a_k\langle \alpha,H_k\rangle\geq 0, \ \forall\, \alpha\in (\Sigma')^+\}.
\end{eqnarray*}

From the description (\ref{eq:sigma-reduit}) of $(\Sigma')^+$ we finally found that 
$$
\agot_+=\sum_{k=1}^r \R^{\geq 0}(X_1+\cdots+X_k).
$$

\subsection{Appendix C}\label{sec:appendix-C}

Let $\varphi: H\to K$ be a morphism of compact connected Lie group. Let $d\varphi:\hgot\to\kgot$ be the 
corresponding morphism of Lie algebras. Any Hamiltonian $K$-manifold 
$(M,\Phi_K)$ can be seen as a Hamiltonian $H$-manifold, with moment map $\Phi_H=d\varphi^*\circ\Phi_K$.

The morphism $\varphi$ induces a map $\varphi^*: R(K)\to R(H)$. When $E\in\Rfor(K)$ is 
$H$-admissible (see Definition \ref{def:admis}), one can define its 
``restriction'' to $H$, that we denoty by $\varphi^* E$ (or simply $E\vert_H$).

The aim of this appendix is to check that the following version of \textbf{[P2]} holds.

\begin{prop} Let $M$ be a pre-quantized \emph{proper} Hamiltonian $K$-manifold. 
If $M$ is still \emph{proper} as a
Hamiltonian $H$-manifold. Then $\qfor_K(M)$ is $H$-admissible and we
have the following equality in $\Rfor(H)$ :
$$
\qfor_K(M)\vert_H=\qfor_H(M).
$$
\end{prop}

\begin{proof} The proof is given in \cite{pep-formal} when $\varphi$ is the inclusion of a 
subgroup. Let us generalize this result to a general morphism $\varphi:H\to K$. Let $L:=\varphi(H)$.
We write $\varphi=i\circ j$ where $i:L\croc K$ is the one to one map given by the inclusion, 
and $j:H\to L$ is the onto morphism induced by $\varphi$. 

We consider the one to one linear map  $j^*:\lgot^*\to \hgot^*$. We can choose compatible system of positive roots 
for $H$ and $L$, so that $j^*$ defines a one to one map from $\what{L}$ to $\what{H}$. Then $j^*V^L_\mu= V^H_{j^*(\mu)}$
for any highest weight $\mu\in \what{L}$.

Let $M$ be a \emph{proper} Hamiltonian $K$-manifold which is prequantized by a line bundle $\Lcal$.  
Since $j:H\to L$ is  onto we have:

$\bullet$ Any $E\in \Rfor(K)$ is $H$-admissible if and only if $E$ is $L$-admissible, and $E\vert_H=j^*(E\vert_L)$.

$\bullet$ $M$ is \emph{proper} as a Hamiltonian $H$-manifold if and only if it is \emph{proper} as a
Hamiltonian $L$-manifold.

Hence 
\begin{eqnarray*}
\qfor_K(M)\vert_H=j^*(\qfor_K(M)\vert_L)&=&j^*\left(\qfor_L(M)\right)\\
&=& j^*\Big(\sum_{\mu\in \what{L}}\Qcal(M_{\mu,L})\, V_{\mu}^{L}\Big)\\
&=& \sum_{\mu\in \what{L}}\Qcal(M_{\mu,L})\, V_{j^*(\mu)}^{H},
\end{eqnarray*}
where $M_{\mu,L}$ is the symplectic reduction at $\mu$ relative to the action of $L$ on $M$. Our proof 
is then finished if we check that 
\begin{equation}\label{eq:Qcal-H-L}
\Qcal(M_{\mu,L})=\Qcal(M_{j^*(\mu),H})
\end{equation}
holds for any $\mu\in\what{L}$.

The one to one map $j^*:\lgot^*\to\hgot^*$ satisfies $h\cdot j^*(\xi)=j^*(j(h)\cdot\xi)$ for any 
$h\in H$ and $\xi\in\lgot^*$. Hence the map $j^*$ defines a $\varphi$-equivariant symplectomorphism between 
the coadjoint orbits $L\cdot\xi$ and $H\cdot j^*(\xi)$.

Let $\mu\in\what{L}$. We work now with the {\em proper} Hamiltonian $L$-manifold $\Xcal:=M\times\overline{L\cdot\mu}$ 
which is prequantized by the line bundle $\Lcal_\Xcal:=\Lcal\otimes\C_{[-\mu]}$.  Let 
$\Phi_L: M\times\overline{L\cdot\mu}\to \lgot^*$ be the moment map relative to the $L$-action. Let $\Hcal_L$ be the  
Hamiltonian vector field of the function $\frac{-1}{2}\parallel\Phi_L\parallel^2$.

The ``pushed'' Thom symbol $\Thom^{\Hcal_L}(\Xcal)$ 
is $L$-transversaly elliptic when whe restrict it to a $L$-invariant relatively compact open subset $\Ucal$ such that 
$$
\partial\Ucal\cap \Cr(\parallel\Phi_L\parallel^2)=\emptyset.
$$
Then we may consider the equivariant index $\indice^L_\Ucal(\Thom^{\Hcal_L}(\Xcal)\vert_\Ucal\otimes\Lcal_\Xcal)$.
We know from Theorem  \ref{theo:loc-index} that
\begin{equation}\label{eq:Q-mu-L}
\Big[\indice^L_\Ucal(\Thom^{\Hcal_L}(\Xcal)\vert_\Ucal\otimes\Lcal_\Xcal)\Big]^L=\Qcal(M_{\mu,L})
\end{equation}
when $\Phi_L^{-1}(0)\subset\Ucal$.

Now we look at $\Xcal$ as a Hamiltonian $H$-manifold through the onto morphism 
$j:H\to L$: then $\Xcal\simeq M\times\overline{H\cdot j^*(\mu)}$. Let $\Phi_H=j^*\circ \Phi_L$ be the 
cooresponding moment map. Since $j^*$ is one to one, the functions $\parallel\Phi_L\parallel^2$ 
and $\parallel\Phi_H\parallel^2$ coincides if we choose appropriate invariant scalar products on 
$\lgot^*$ and $\hgot^*$. Then we have $\Phi_L^{-1}(0)=\Phi_H^{-1}(0)$ and $\Hcal_L=\Hcal_H$. 
As before Theorem  \ref{theo:loc-index} gives 
\begin{equation}\label{eq:Q-mu-H}
\Big[\indice^H_\Ucal(\Thom^{\Hcal_H}(\Xcal)\vert_\Ucal\otimes\Lcal_\Xcal)\Big]^H=\Qcal(M_{j^*(\mu),H}).
\end{equation}

Since $[E]^L=[j^*E]^H$ for any $E\in\Rfor(L)$, the relations (\ref{eq:Q-mu-L}) and  (\ref{eq:Q-mu-H}) imply
finally (\ref{eq:Qcal-H-L}).

\end{proof}


\end{document}